\documentclass[12pt]{amsart}
\usepackage{amscd,amsmath,amsthm,amssymb}
\usepackage{pstcol,pst-plot,pst-3d}%\usepackage[T1]{fontenc}
\usepackage{lmodern,pst-node}%\usepackage{geometric}
\usepackage{multicol}
\usepackage{epic,eepic}
\usepackage{amsfonts,amssymb,amscd,amsmath,enumerate,verbatim}
\psset{unit=0.7cm,linewidth=0.8pt,arrowsize=2.5pt 4}
\newpsstyle{fatline}{linewidth=1.5pt}
\newpsstyle{fyp}{fillstyle=solid,fillcolor=verylight}
\definecolor{verylight}{gray}{0.97}
\definecolor{light}{gray}{0.9}
\definecolor{medium}{gray}{0.85}
\definecolor{dark}{gray}{0.6}

\def\NZQ{\Bbb}               % the font for N,Z,Q,R,C
\def\NN{{\NZQ N}}

\def\ZZ{{\NZQ Z}}

\def\FF{{\NZQ F}}

\def\frk{\frak}               % font for "Fraktur"

\def\mm{{\frk m}}

\def\Phi{{\frk n}}
\def\Phi{{\frk N}}
\def\MP{{\mathcal P}}
\def\MQ{{\mathcal Q}}

\def\opn#1#2{\def#1{\operatorname{#2}}} % to make operators
\opn\chara{char} \opn\length{\ell} \opn\pd{pd} \opn\rk{rk}
\opn\projdim{proj\,dim} \opn\injdim{inj\,dim} \opn\rank{rank}
\opn\depth{depth} \opn\grade{grade} \opn\height{height}
\opn\embdim{emb\,dim} \opn\codim{codim}

\opn\Tr{Tr} \opn\bigrank{big\,rank}
\opn\superheight{superheight}\opn\lcm{lcm}
\opn\trdeg{tr\,deg}%\emph{
\opn\reg{reg} \opn\lreg{lreg} \opn\ini{in} \opn\lpd{lpd}
\opn\size{size}\opn\bigsize{bigsize}
\opn\cosize{cosize}\opn\bigcosize{bigcosize}
\opn\sdepth{sdepth}\opn\sreg{sreg}
\opn\link{link}\opn\fdepth{fdepth}
\opn\div{div} \opn\Div{Div} \opn\cl{cl} \opn\Cl{Cl}
\opn\Spec{Spec} \opn\Supp{Supp} \opn\supp{supp} \opn\Sing{Sing}
\opn\Ass{Ass} \opn\Min{Min}\opn\Mon{Mon} \opn\dstab{dstab} \opn\astab{astab}
\opn\Syz{Syz}
\opn\Ann{Ann} \opn\Rad{Rad} \opn\Soc{Soc}
\opn\Im{Im} \opn\Ker{Ker} \opn\Coker{Coker} \opn\Am{Am}
\opn\Hom{Hom} \opn\Tor{Tor} \opn\Ext{Ext} \opn\End{End}
\opn\Aut{Aut} \opn\id{id}

\opn\nat{nat}
\opn\pff{pf}%   \pf exists already
\opn\Pf{Pf} \opn\GL{GL} \opn\SL{SL} \opn\mod{mod} \opn\ord{ord}
\opn\Gin{Gin} \opn\Hilb{Hilb}\opn\sort{sort}\opn\width{width}
\opn\aff{aff} \opn\con{conv} \opn\relint{relint} \opn\st{st}
\opn\lk{lk} \opn\cn{cn} \opn\core{core} \opn\vol{vol}
\opn\link{link} \opn\star{star}\opn\lex{lex}
\opn\gr{gr}

\def\pot#1#2{#1[\kern-0.28ex[#2]\kern-0.28ex]}

\opn\dirlim{\underrightarrow{\lim}}
\opn\inivlim{\underleftarrow{\lim}}
\let\union=\cup

\let\tensor=\otimes
\let\iso=\cong

\let\Dirsum=\bigoplus

\let\to=\rightarrow

\def\Implies{\ifmmode\Longrightarrow \else
        \unskip${}\Longrightarrow{}$\ignorespaces\fi}
\def\implies{\ifmmode\Rightarrow \else
        \unskip${}\Rightarrow{}$\ignorespaces\fi}
\def\iff{\ifmmode\Longleftrightarrow \else
        \unskip${}\Longleftrightarrow{}$\ignorespaces\fi}

\let\:=\colon
\newtheorem{Theorem}{Theorem}[section]
 \newtheorem{Lemma}[Theorem]{Lemma}
 \newtheorem{Corollary}[Theorem]{Corollary}
 \newtheorem{Proposition}[Theorem]{Proposition}

\let\epsilon\varepsilon
\let\kappa=\varkappa
\def\qed{\ifhmode\textqed\fi
      \ifmmode\ifinner\quad\qedsymbol\else\dispqed\fi\fi}
\def\textqed{\unskip\nobreak\penalty50
       \hskip2em\hbox{}\nobreak\hfil\qedsymbol
       \parfillskip=0pt \finalhyphendemerits=0}
\def\dispqed{\rlap{\qquad\qedsymbol}}

\opn\dis{dis}
\def\pnt{{\raise0.5mm\hbox{\large\bf.}}}

\opn\Lex{Lex}

\begin{document}
 \title {Linearly related polyominoes}

 \author {Viviana Ene, J\"urgen Herzog, Takayuki Hibi}

\address{Viviana Ene, Faculty of Mathematics and Computer Science, Ovidius University, Bd.\ Mamaia 124,
 900527 Constanta, Romania, and
 \newline
 \indent Simion Stoilow Institute of Mathematics of the Romanian Academy, Research group of the project  ID-PCE-2011-1023,
 P.O.Box 1-764, Bucharest 014700, Romania} \email{vivian@univ-ovidius.ro}

\address{J\"urgen Herzog, Fachbereich Mathematik, Universit\"at Duisburg-Essen, Campus Essen, 45117
Essen, Germany} \email{juergen.herzog@uni-essen.de}

\address{Takayuki Hibi, Department of Pure and Applied Mathematics, Graduate School of Information Science and Technology,
Osaka University, Toyonaka, Osaka 560-0043, Japan}
\email{hibi@math.sci.osaka-u.ac.jp}

\thanks{The first author was supported by the grant UEFISCDI,  PN-II-ID-PCE- 2011-3-1023.}

 \begin{abstract}
We classify all  convex polyomino ideals which are  linearly related or have a linear resolution. Convex stack polyominoes  whose ideals are extremal Gorenstein are also classified. In addition, we characterize, in combinatorial terms, the distributive lattices  whose join-meet ideals  are extremal Gorenstein or have a linear resolution. 
 \end{abstract}

\thanks{}

 \maketitle

\section*{Introduction}
The ideal of inner minors of a polyomino, a so-called polyomino ideal,  is  generated by  certain subsets of $2$-minors of an $m\times n$-matrix $X$ of
indeterminates. Such ideals have first been studied by Qureshi in \cite{Q}. They include the two-sided ladder  determinantal ideals of $2$-minors  which may also
be viewed as the join-meet ideal of a planar distributive lattice. It is a challenging problem to understand the graded free resolution of such ideals. In
\cite{ERQ}, Ene, Rauf and Qureshi succeeded to compute the regularity of such joint-meet ideals. Sharpe \cite{S1, S2} showed that the  ideal $I_2(X)$  of all $2$-minors  of $X$ is
linearly related, which means that   $I_2(X)$  has linear relations. Moreover, he described these relations explicitly and conjectured that also the ideals of
$t$-minors $I_t(X)$  are generated by a certain type of linear relations. This  conjecture was then proved by Kurano \cite{K}. In the case that the base field over
which $I_t(X)$ is defined contains the rational numbers, Lascoux \cite{L} gives the explicit free resolution of all ideals of $t$-minors. Unfortunately, the
resolution of $I_t(X)$ in general  may depend on the characteristic of the base field.  Indeed, Hashimoto \cite{H} showed that for  $2 \leq t \leq  \min(m, n)-3$,
the second  Betti number $\beta_2$  of $I_t(X)$ depends on the characteristic. On the other hand, by using  squarefree divisor complexes \cite{BH} as introduced by Bruns and the second author of this paper, it follows from
\cite[Theorem 1.3]{BH}  that $\beta_2$ for $t=2$ is independent of the characteristic.

In  this paper we  use as a main tool  squarefree divisor complexes to study the first syzygy module  of a polyomino ideal. In particular, we classify all convex polyominoes which are linearly related; see Theorem~\ref{main}. This is the main result of this paper. In the first section we recall the concept of polyomino ideals and show that the polyomino ideal of a convex polyomino has a quadratic Gr\"obner basis. The second section of the paper is devoted to state and to prove Theorem~\ref{main}. As mentioned before,  the proof heavily depends on the theory of squarefree divisor complexes which allow to compute the multi-graded Betti numbers of a toric ideal. To apply this theory, one observes that the polyomino ideal of a convex polyomino may be naturally identified with a toric ideal.   The crucial conclusion deduced from this observation, formulated in Corollary~\ref{inducedsubpolyomino}, is then that the Betti numbers of a polyomino ideal is bounded below by the Betti numbers of the polyomino ideal of any induced subpolyomino. Corollary~\ref{inducedsubpolyomino} allows to reduce the study  of the relation of polyomino ideals  to that of a finite number of polyominoes with a small number of cells which all can be analyzed by the use of a computer algebra system.

In the last section, we classify all convex polyominoes whose polyomino ideal  has a linear resolution (Theorem~\ref{linear}) and all convex stack polyominoes whose polyomino ideal is extremal Gorenstein (Theorem~\ref{stack}). Since polyomino ideals overlap with join-meet ideals, it is of interest which of the ideals among the join-meet ideals  have a linear resolution or are extremal Gorenstein. The answers are given in Theorem~\ref{hibione} and Theorem~\ref{hibitwo}. It turns out that the classifications for both classes of ideals almost lead  to the same result.

\section{Polyominoes}

In this section we consider  polyomino ideals. This class of ideals of $2$-minors was introduced   by Qureshi \cite{Q}.
To this end, we consider on $\NN^2$ the natural partial order defined as follows: $(i,j) \leq (k,l)$ if and only if $i \leq k$ and $j \leq l$.
The set $\NN^2$ together with this partial order is a distributive lattice.

If $a,b \in \NN^2$ with $a \leq b$,
then the set $[a,b]= \{ c \in \NN^2|\; a \leq c \leq b\}$ is an interval of $\NN^2$. The interval $C=[a,b]$ with $b=a+(1,1)$ is called a {\em cell}  of $\NN^2$.
The elements of $C$ are called the {\em vertices} of $C$ and  $a$ is called the {\em left lower corner} of $C.$ The {\em egdes} of the cell $C$ are the sets $\{a,
(a+(1,0)\}, \{a,a+(0,1)\},  \{(a+(1,0),  a+(1,1)\}$ and   $\{(a+(0,1),  a+(1,1)\}$.

Let $\MP$ be a finite collection of cells and $C,D\in \MP$. 
Then $C$ and $D$ are {\em connected}, if there is a sequence of
 cells of $\MP$ given by $C= C_1, \ldots, C_m =D$ such that $C_i \cap C_{i+1}$ is an edge of $C_i$ for $i=1, \ldots, m-1$. If, in addition, $C_i \neq C_j$ for all
 $i \neq j$, then $\mathcal{C}$ is called a {\em path} (connecting $C$ and $D$). The collection of cells $\MP$ is called a {\em polyomino} if any two cells of
 $\MP$ are connected; see Figure~\ref{polyomino}. The set of vertices of $\MP$, denoted $V(\MP)$, is the union of the vertices of all cells belonging to $\MP$. Two polyominoes are called {\em isomorphic}  if they are mapped to each other by a composition of translations, reflections and rotations.

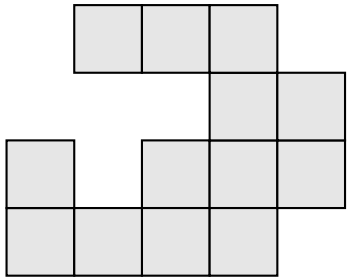
\begin{figure}[hbt]
\begin{center}
\psset{unit=0.9cm}
\begin{pspicture}(4.5,-1)(4.5,3.5)
{
\pspolygon[style=fyp,fillcolor=light](4,0)(4,1)(5,1)(5,0)
\pspolygon[style=fyp,fillcolor=light](5,0)(5,1)(6,1)(6,0)
\pspolygon[style=fyp,fillcolor=light](3,2)(3,3)(4,3)(4,2)
\pspolygon[style=fyp,fillcolor=light](5,1)(5,2)(6,2)(6,1)
\pspolygon[style=fyp,fillcolor=light](4,2)(4,3)(5,3)(5,2)
\pspolygon[style=fyp,fillcolor=light](5,2)(5,3)(6,3)(6,2)
\pspolygon[style=fyp,fillcolor=light](6,0)(6,1)(7,1)(7,0)
\pspolygon[style=fyp,fillcolor=light](6,1)(6,2)(7,2)(7,1)
\pspolygon[style=fyp,fillcolor=light](4,-1)(4,0)(5,0)(5,-1)
\pspolygon[style=fyp,fillcolor=light](5,-1)(5,0)(6,0)(6,-1)
\pspolygon[style=fyp,fillcolor=light](2,-1)(2,0)(3,0)(3,-1)
\pspolygon[style=fyp,fillcolor=light](3,-1)(3,0)(4,0)(4,-1)
\pspolygon[style=fyp,fillcolor=light](2,-0)(2,1)(3,1)(3,0)
}
\end{pspicture}
\end{center}
\caption{A  polyomino}\label{polyomino}
\end{figure}

We call a polyomino $\MP$ {\em row convex}, if for any two cells $C,D$ of $\MP$ with left  lower corner $a=(i,j)$ and $b=(k,j)$ respectively, and such that $k>i$,  it
follows that all cells with left lower corner $(l,j)$ with $i\leq l\leq k$ belong to $\MP$.  Similarly, one defines {\em column convex} polyominoes.  The
polyomino $\MP$ is called {\em convex} if it is row and column convex.

The polyomino displayed in Figure~\ref{polyomino}  is not convex, while  Figure~\ref{convex} shows a convex polyomino. Note that a convex polyomino is not convex
in the common geometric sense.

\begin{figure}[hbt]
\begin{center}
\psset{unit=0.9cm}
\begin{pspicture}(4.5,-1)(4.5,3.5)
{
\pspolygon[style=fyp,fillcolor=light](2.8,0)(2.8,1)(3.8,1)(3.8,0)
\pspolygon[style=fyp,fillcolor=light](3.8,-1)(3.8,0)(4.8,0)(4.8,-1)
\pspolygon[style=fyp,fillcolor=light](4.8,0)(4.8,1)(5.8,1)(5.8,0)
\pspolygon[style=fyp,fillcolor=light](3.8,1)(3.8,2)(4.8,2)(4.8,1)
\pspolygon[style=fyp,fillcolor=light](3.8,0)(3.8,1)(4.8,1)(4.8,0)
}
\end{pspicture}
\end{center}
\caption{A convex polyomino}\label{convex}
\end{figure}
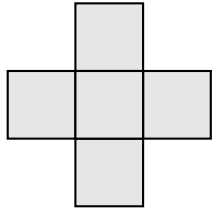

\medskip
Now let $\MP$ be any collection of cells. We may assume that the vertices of all the cells of $\MP$ belong to the interval $[(1,1),(m,n)]$.  Fix a field $K$ and let $S$ be
the polynomial ring over $K$ in the variables $x_{ij}$ with $(i,j)\in \MP$. The {\em ideal of inner
minors} $I_\MP\subset S$ of $\MP$, is the ideal generated by all $2$-minors $x_{il}x_{kj}-x_{kl}x_{ij}$ for which  $[(i,j),(k,l)]\subset V(\MP)$. Furthermore, we denote by
$K[\MP]$ the $K$-algebra $S/I_\MP$. If $\MP$ happens to be a polyomino, then $I_\MP$  will also be  called a {\em polyomino ideal}.

For example, the polyomino $\MP$ displayed in Figure~\ref{convex} may be embedded into the interval $[(1,1),(4,4)]$. Then, in these coordinates,  $I_\MP$ is
generated by the $2$-minors
\begin{eqnarray*}
&& x_{22}x_{31}-x_{32}x_{21}, x_{23}x_{31}-x_{33}x_{21},x_{24}x_{31}-x_{34}x_{21}, x_{23}x_{32}-x_{33}x_{22},\\
&& x_{24}x_{32}-x_{34}x_{22}, x_{24}x_{33}-x_{34}x_{23},
x_{13}x_{22}-x_{12}x_{23}, x_{13}x_{32}-x_{12}x_{33}, \\
&& x_{13}x_{42}-x_{12}x_{43}, x_{23}x_{42}-x_{22}x_{43},  x_{33}x_{42}-x_{32}x_{43}.
\end{eqnarray*}

The following result has been shown by Qureshi in \cite[Theorem 2.2]{Q}.

\begin{Theorem}
\label{ayesha}
Let $\MP$ be a convex polyomino. Then $K[\MP]$ is a normal Cohen--Macaulay domain.
\end{Theorem}

The proof of this theorem is based on the fact that $I_\MP$ may be viewed as follows as a toric ideal: with the assumptions and notation as introduced before, we
may assume that $V(\MP)\subset [(1,1),(m,n)]$. Consider the $K$-algebra homomorphism $\varphi\: S\to T$ with $\varphi(x_{ij})=s_it_j$ for all $(i,j)\in V(\MP)$.
Here $T=K[s_1,\ldots,s_m, t_1,\ldots,t_n]$ is the polynomial ring over $K$ in the variables $s_i$ and $t_j$. Then, as observed by Qureshi, $I_\MP=\Ker \varphi$.
It follows that $K[\MP]$ may be identified with the edge ring of the bipartite graph $G_\MP$ on the vertex set $\{s_1,\ldots,s_m\}\union\{t_1,\ldots,t_n\}$ and
edges $\{s_i,t_j\}$ with $(i,j)\in V(\MP)$. With this interpretation of $K[\MP]$ in mind and by using \cite{HO}, we obtain

\begin{Proposition}
\label{hibiohsugi}
Let $\MP$ be a convex polyomino. Then $I_\MP$ has a quadratic Gr\"obner basis.
\end{Proposition}

\begin{proof} We use the crucial fact, proved in \cite{HO}, that the toric ideal which defines the edge ring of a bipartite graph has a quadratic Gr\"obner basis if
and only if each $2r$-cycle with $r\geq 3$ has a chord.  By what we explained before,  a $2k$-cycle, after identifying the vertices of $\MP$ with the edges of a
bipartite graph, is nothing but a sequence of vertices $a_1,\ldots, a_{2r}$ of $\MP$ with
\[
a_{2k-1}= (i_k,j_k) \quad \text{and}\quad a_{2k}= (i_{k+1}, j_k) \quad\text{for $k=1,\ldots,r$}
\]
such that  $i_{r+1}=i_1$,  $i_k\neq i_\ell$ and   $j_k\neq j_\ell$ for all $k,\ell\leq r$ and $k\neq \ell$.

A typical such sequence of  pairs of integers is the following:
\begin{eqnarray*}
&3 2 2 4 4 5 5 3 &\\
&1 1 3 3 2 2 4 4&
\end{eqnarray*}
Here the first row is the sequence of the first component and the second row the sequence of the second component  of the  vertices $a_i$. This pair of sequences
represents an $8$-cycle. It follows from Lemma~\ref{goodforstudents} that there exist  integers  $s$ and $t$  with $1\leq t,  s\leq r$  and $t\neq s,s+1$ such
that either $i_s<i_t<i_{s+1}$ or  $i_{s+1}<i_t<i_{s}$. Suppose that   $i_s<i_t<i_{s+1}$. Since $a_{2s-1}=(i_s,j_s)$ and $a_{2s}=(i_{s+1},j_s)$  are vertices of
$\MP$ and since  $\MP$ is convex, it follows that $(i_t,j_s)\in \MP$. This vertex corresponds to a chord of the cycle $a_1,\ldots, a_{2r}$. Similarly one argues
if $i_{s+1}<i_t<i_{s}$.
\end{proof}

\begin{Lemma}
\label{goodforstudents}
Let $r \geq 3$ be an integer and $f:[r+1]\to \ZZ$ a function such that
$f(i) \neq f(j)$ for $1 \leq i < j \leq r$ and $f(r+1) = f(1)$.
Then there exist $1 \leq s, \, t \leq r$ such that
one has either $f(s) < f(t) < f(s+1)$ or $f(s+1) < f(t) < f(s)$.
\end{Lemma}

\begin{proof}
Let, say, $f(1) < f(2)$.  Since $f(r+1) = f(1)$,
there is $2 \leq q \leq r$ with
\[
f(1) < f(2) < \cdots < f(q) > f(q + 1).
\]
\begin{itemize}
\item
Let $q = r$.  Then, since $q = r \geq 3$, one has $(f(1) =) \, f(r+1)
< f(2) < f(r)$.
\item
Let $q < r$ and $f(q+1) > f(1)$.
Since $f(q+1) \not\in \{f(1), f(2), \ldots, f(q) \}$,
it follows that
there is $1 \leq s < q$ with $f(s) < f(q) < f(s+1)$.
\item
Let $q < r$ and $f(q+1) < f(1)$.
Then one has $f(q+1) < f(1) < f(q)$.
\end{itemize}
The case of $f(1) > f(2)$ can be discussed similarly.
\end{proof}

We denote the graded Betti numbers of $I_\MP$ by $\beta_{ij}(I_\MP)$.

\begin{Corollary}
\label{no}
Let $\MP$ be a convex polyomino. Then $\beta_{1j}(I_\MP)=0$ for $j>4$.
\end{Corollary}

\begin{proof}
By Proposition~\ref{hibiohsugi}, there exists  a monomial order $<$ such that $\ini_<(I_\MP)$ is generated in degree 2. Therefore, it follows  from \cite[Corollary 4]{HS} that
$\beta_{1j}(\ini_<(I_\MP))=0$ for $j>4$. Since $\beta_{1j}(I_\MP)\leq \beta_{1j}(\ini_<(I_\MP))$ (see, for example, \cite[Corollary 3.3.3]{HH}), the desired conclusion follows.
\end{proof}

\section{The first syzygy module of a polyomino ideal}

Let $\MP$ be a convex polyomino and let $f_1,\ldots,f_m$ be the minors generating $I_\MP$. In this section we study the relation module $\Syz_1(I_\MP)$ of
$I_\MP$ which is the kernel of the $S$-module homomorphism $\Dirsum_{i=1}^mSe_i\to I_\MP$ with $e_i\mapsto f_i$ for $i=1,\ldots,m$.  The graded module
$\Syz_1(I_\MP)$ has generators in degree $3$ and no generators in degree $>4$, as we have seen in Corollary~\ref{no}. We say that $I_\MP$ (or simply $\MP$) is {\em
linearly related} if $\Syz_1(I_\MP)$ is generated only in degree $3$.

Let $f_i$ and $f_j$ be two distinct generators of $I_\MP$. Then the Koszul relation $f_ie_j-f_je_i$ belongs  $\Syz_1(I_\MP)$. We call $f_i,f_j$ a {\em Koszul
relation pair} if $f_ie_j-f_je_i$ is a minimal generator of $\Syz_1(I_\MP)$. The main result of this section is the following.

\begin{Theorem}
\label{main}
Let $\MP$ be a convex polyomino. The following conditions are equivalent:
\begin{enumerate}
\item[{\em (a)}] $\MP$ is linearly related;
\item[{\em (b)}] $I_\MP$ admits no Koszul relation pairs;
\item[{\em (c)}] Let, as we may assume, $[(1,1),(m,n)]$ be the smallest interval with the property that $V(\MP)\subset [(1,1),(m,n)]$. We refer to the elements $(1,1), (m,1), (1,n)$ and $(m,n)$ as the corners. Then $\MP$ has the shape as displayed in Figure~\ref{shape}, and one of the following conditions hold:
    \begin{enumerate}
    \item[{\em (i)}] at most one of the  corners does not belong to $V(\MP)$;
    \item[{\em (ii)}] two of the corners do not belong to $V(\MP)$,  but they are not opposite to each other. In other words, the missing corners are not   the corners  $(1,1),(n,m)$, or the corners $(m,1),(1,n)$.
\item[{\em (iii)}] three of the corners do not belong to $V(\MP)$. If the missing corners are $(m,1),(1,n)$ and $(m,n)$ (which one may assume without loss  of generality), then referring to Figure~\ref{shape} the following conditions must be satisfied: either $i_2=m-1$ and $j_4\leq j_2$, or $j_2=n-1$ and $i_4\leq i_2$.
    \end{enumerate}
\end{enumerate}
\end{Theorem}

As an essential tool in the proof of this theorem we recall the co-called squarefree divisor complex, as introduced in \cite{HH}. Let $K$ be field,  $H\subset
\NN^n$ an affine semigroup and $K[H]$ the semigroup ring attached to it. Suppose that $h_1,\ldots, h_m\in \NN^n$ is the unique minimal set of generators of $H$.
We consider the polynomial ring $T=K[t_1,\ldots,t_n]$ in the variables $t_1,\ldots,t_n$. Then $K[H]=K[u_1,\ldots,u_m]\subset T$ where
$u_i=\prod_{j=1}^nt_j^{h_i(j)}$ and where $h_i(j)$ denotes the $j$th component of the integer vector $h_i$. We choose a presentation $S=K[x_1,\ldots,x_m]\to K[H]$
with $x_i\mapsto u_i$ for $i=1,\ldots,m$. The kernel $I_H$ of this $K$-algebra homomorphism is called the toric ideal of $H$. We assign a $\ZZ^n$-grading to $S$
by setting $\deg x_i=h_i$. Then $K[H]$ as well as $I_H$ become $\ZZ^n$-graded $S$-modules. Thus $K[H]$ admits a minimal $\ZZ^n$-graded  $S$-resolution $\FF$ with
$F_i=\Dirsum_{h\in H} S(-h)^{\beta_{ih}(K[H])}$.

In the case that all $u_i$ are monomials of the same degree, one can assign to  $K[H]$ the structure of a standard graded $K$-algebra by setting  $\deg u_i=1$ for all $i$. The degree of $h$ with respect to this standard grading will  be denoted $|h|$.

Given $h\in H$, we define the {\em squarefree divisor complex} $\Delta_h$ as follows: $\Delta_h$ is the simplicial complex  whose faces $F=\{i_1,\ldots,i_k\}$ are
the subsets of $[n]$ such that $u_{i_1}\cdots u_{i_k}$ divides $t_1^{h(1)}\cdots t_n^{h(n)}$ in $K[H]$. We denote by $\tilde{H}_{i}(\Gamma, K)$ the $i$th reduced
simplicial homology of a simplicial complex $\Gamma$.

\begin{Proposition}[Bruns-Herzog \cite{BH}]
\label{bh}
With the notation and assumptions introduced one has  $\Tor_i(K[H],K)_h\iso\tilde{H}_{i-1}(\Delta_h, K)$. In particular,
$$\beta_{ih}(K[H])=\dim_K\tilde{H}_{i-1}(\Delta_h, K).$$
\end{Proposition}

Let $H'$ be a subsemigroup of $H$ generated by a subset of the set of generators  of $H$, and let $S'$ be the polynomial ring over $K$ in the variables $x_i$ with $h_i$ generator of $ H^\prime$. Furthermore, let $\FF'$ the $\ZZ^{n}$-graded  free $S'$-resolution  of $K[H']$. Then, since $S$ is a flat $S'$-module,  $\FF'\tensor_{S'}S$ is a $\ZZ^n$-graded free $S$-resolution of $S/I_H'S$. The inclusion $K[H']\to K[H]$ induces a $\ZZ^n$-graded complex homomorphism $\FF'\tensor_{S'}S\to \FF$. Tensoring this complex homomorphism with $K=S/\mm$, where $\mm$ is the graded maximal ideal of $S$,  we obtain the following sequence of isomorphisms and  natural maps of $\ZZ^n$-graded $K$-modules
\[
\Tor_i^{S'}(K[H'],K)\iso H_i(\FF'\tensor_{S'}K)\iso H_i(\FF'\tensor_{S'}S)\tensor_SK)\to H_i(\FF\tensor_SK)\iso \Tor_i^S(K[H],K).
\]

For later applications we need

\begin{Corollary}
\label{refinement}
With the notation and assumptions introduced, let  $H'$ be a subsemigroup of $H$ generated by a subset of the set of generators  of $H$, and let  $h$ be an
element of $H'$ with the property  that $h_i\in H'$ whenever  $h-h_i\in H$. Then  the natural $K$-vector space homomorphism $\Tor_i^{S'}(K[H'],K)_h\to
\Tor_i^S(K[H],K)_h$ is an isomorphism for all $i$.
\end{Corollary}

\begin{proof}
Let $\Delta_h'$ be the squarefree divisor complex of $h$ where $h$ is viewed as an element of $H'$. Then we obtain the following commutative diagram
\begin{eqnarray*}
\label{diagram}
\begin{CD}
\Tor_i(K[H'],K)_h @>>> \Tor_i(K[H],K)_h\\
@VVV @VVV\\
\tilde{H}_{i-1}(\Delta_h', K)@>>> \tilde{H}_{i-1}(\Delta_h, K).
\end{CD}
\end{eqnarray*}
The vertical maps are isomorphisms, and also the lower horizontal map is an isomorphism, simply because $\Delta_h'=\Delta_h$, due to assumptions on $h$. This
yields the desired conclusion.
\end{proof}

Let $H\subset \NN^n$ be an affine semigroup generated by $h_1,\ldots, h_m$. An affine subsemigroup $H'\subset H$ generated by  a subset of $\{h_1,\ldots, h_m\}$ will be called a {\em homological pure} subsemigroup of $H$ if for all $h\in H'$ and all $h_i$ with $h-h_i\in H$  it follows that $h_i\in H'$.

\medskip
 As an immediate   consequence of Corollary~\ref{refinement} we  obtain

\begin{Corollary}
\label{homologicallypure}
Let $H'$ be a homologically pure subsemigroup  of $H$. Then $$\Tor_i^{S'}(K[H'],K)\to  \Tor_i^S(K[H],K)$$   is injective for all $i$. In other words, if $\FF'$ is the minimal $\ZZ^n$-graded free $S'$-resolution of $K[H']$ and $\FF$ is the minimal $\ZZ^n$-graded free $S$-resolution of $K[H]$, then the complex homomorphism
$\FF'\tensor S\to \FF$ induces an injective map $\FF'\tensor K\to \FF\tensor K$. In particular, any minimal set of generators of $\Syz_i(K[H'])$ is part of a minimal set of generators  of  $\Syz_i(K[H])$. Moreover, $\beta_{ij}(I_{H'})\leq \beta_{ij}(I_H)$ for all $i$ and $j$.
\end{Corollary}

\medskip
We fix a field $K$ and  let $\MP\subset [(1,1),(m,n)]$  be a convex polyomino. Let as before $S$ be the polynomial ring over $K$ in
the variables $x_{ij}$ with $(i,j)\in V(\MP)$ and $K[\MP]$ the $K$-subalgebra of the polynomial ring $T=K[s_1,\ldots,s_m,t_1,\ldots,t_n]$ generated by the
monomials $u_{ij}=s_it_j$ with $(i,j)\in V(\MP)$. Viewing  $K[\MP]$ as a semigroup ring $K[H]$, it is convenient to  identify the semigroup elements with the monomial they represent.

Given sets $\{i_1,i_2,\ldots, i_s\}$ and $\{j_1,j_2,\ldots, j_t\}$  of integers with $i_k\subset [m]$ and $j_k\subset [n]$ for all $k$, we  let $H'$ be the subsemigroup of $H$ generated by the elements $s_{i_k}t_{j_l}$  with $(i_k,j_l)\in V(\MP)$. Then $H'$  a homologically pure subsemigroup of $H$.  Note that $H'$ is also a combinatorially pure subsemigroup of $H$ in the sense of \cite{HHO}.

A collection of cells  $\MP'$ will be called a {\em collection of cells}  of $\MP$ {\em induced}  by the columns  $i_1,i_2,\ldots, i_s$ and the rows   $j_1,j_2,\ldots, j_t$, if the following holds:  $(k,l)\in V(\MP')$ if and only if $(i_k,j_l)\in V(\MP)$. Observe that $K[\MP']$ is always a domain, since it is a $K$-subalgebra of $K[\MP]$.  The map $V(\MP')\to V(\MP)$, $(k,l)\mapsto (i_k,j_l)$   identifies $I_{\MP'}$ with the  ideal contained in $I_{\MP}$ generated by those $2$-minors of $I(\MP)$ which only involve the variables $x_{i_k,j_l}$.  In the following we always identify $I_{\MP'}$ with this subideal of $I_{\MP}$.

If the induced collection of cells of  $\MP'$ is a polyomino, we call it an {\em induced polyomino}. Any    induced polyomino $\MP'$ of  $\MP$ is again convex.

Consider for example  the polyomino $\MP$  on the left side of Figure~\ref{combinatorial} with left lower corner   $(1,1)$. Then the induced polyomino $\MP'$ shown on the right side of Figure~\ref{combinatorial} is induced by the columns $1,3,4$ and the rows  $1,2,3,4$.

\begin{figure}[hbt]
\begin{center}
\psset{unit=0.9cm}
\begin{pspicture}(4.5,-1)(4.5,3.5)
\rput(-4,0.5){
\pspolygon[style=fyp,fillcolor=light](2.8,-1)(2.8,0)(3.8,0)(3.8,-1)
\pspolygon[style=fyp,fillcolor=light](2.8,0)(2.8,1)(3.8,1)(3.8,0)
\pspolygon[style=fyp,fillcolor=light](3.8,-1)(3.8,0)(4.8,0)(4.8,-1)
\pspolygon[style=fyp,fillcolor=light](4.8,0)(4.8,1)(5.8,1)(5.8,0)
\pspolygon[style=fyp,fillcolor=light](4.8,1)(4.8,2)(5.8,2)(5.8,1)
\pspolygon[style=fyp,fillcolor=light](3.8,1)(3.8,2)(4.8,2)(4.8,1)
\pspolygon[style=fyp,fillcolor=light](3.8,0)(3.8,1)(4.8,1)(4.8,0)
}
\rput(0.3,-1){$\MP$}

\rput(4.5,0.5){
\pspolygon[style=fyp,fillcolor=light](2.8,-1)(2.8,0)(3.8,0)(3.8,-1)
\pspolygon[style=fyp,fillcolor=light](2.8,0)(2.8,1)(3.8,1)(3.8,0)
%\pspolygon[style=fyp,fillcolor=light](3.8,-1)(3.8,0)(4.8,0)(4.8,-1)
%\pspolygon[style=fyp,fillcolor=light](4.8,0)(4.8,1)(5.8,1)(5.8,0)
%\pspolygon[style=fyp,fillcolor=light](4.8,1)(4.8,2)(5.8,2)(5.8,1)
\pspolygon[style=fyp,fillcolor=light](3.8,1)(3.8,2)(4.8,2)(4.8,1)
\pspolygon[style=fyp,fillcolor=light](3.8,0)(3.8,1)(4.8,1)(4.8,0)
}
\rput(8.3,-1){$\MP'$}
\end{pspicture}
\end{center}
\caption{}\label{combinatorial}
\end{figure}

\medskip
Obviously Corollary~\ref{homologicallypure} implies

\begin{Corollary}
\label{inducedsubpolyomino}
Let $\MP'$ be an induced  collection of cells   of $\MP$. Then $\beta_{ij}(I_{\MP'})\leq \beta_{ij}(I_\MP)$ for all $i$ and $j$, and each minimal relation of $I_{\MP'}$ is also a minimal relation of $I_\MP$.
\end{Corollary}

We will now use Corollary~\ref{inducedsubpolyomino} to isolate step by step the linearly related polyominoes.

\begin{Lemma}
\label{begin}
Suppose $\MP$ admits an induced collection of cells $\MP^\prime$  isomorphic to one of those displayed in Figure~\ref{restricted}. Then $I_\MP$ has a Koszul relation pair.
\end{Lemma}

\begin{proof}
We may assume that $V(\MP')\subset [(1,1),(4,4)]$. By using CoCoA~\cite{Co} or Singular~\cite{DGPS} to compute $\Syz_1(I_{\MP'})$ we see that the minors $f_a=[12|12]$ and $f_b=[34|34]$
form a Koszul relation pair of $I_{\MP'}$. Thus the assertion follows from Corollary~\ref{inducedsubpolyomino}.
\end{proof}

\begin{figure}[hbt]
\begin{center}
\psset{unit=0.9cm}
\begin{pspicture}(4.5,-1)(4.5,3.5)

\rput(-3,1){

\pspolygon[style=fyp,fillcolor=light](2.8,-1)(2.8,0)(3.8,0)(3.8,-1)
\pspolygon[style=fyp,fillcolor=light](3.8,-1)(3.8,0)(4.8,0)(4.8,-1)
\pspolygon[style=fyp,fillcolor=light](4.8,0)(4.8,1)(5.8,1)(5.8,0)
\pspolygon[style=fyp,fillcolor=light](4.8,1)(4.8,2)(5.8,2)(5.8,1)
\pspolygon[style=fyp,fillcolor=light](4.8,-1)(4.8,0)(5.8,0)(5.8,-1)

\rput(4.5,-2){(a)}
}

\rput(2.5,1){

\pspolygon[style=fyp,fillcolor=light](2.8,-1)(2.8,0)(3.8,0)(3.8,-1)
\pspolygon[style=fyp,fillcolor=light](3.8,-1)(3.8,0)(4.8,0)(4.8,-1)
\pspolygon[style=fyp,fillcolor=light](4.8,0)(4.8,1)(5.8,1)(5.8,0)
\pspolygon[style=fyp,fillcolor=light](4.8,1)(4.8,2)(5.8,2)(5.8,1)
%\pspolygon[style=fyp,fillcolor=light](4.8,-1)(4.8,0)(5.8,0)(5.8,-1)

\rput(4.5,-2){(b)}
}
\end{pspicture}
\end{center}
\caption{$\MP'$}\label{restricted}
\end{figure}

%\begin{Example}
%{\em The two minors $[14|23]$ and $[23|14]$ of the polyomino $\MP$ shown in Figure~\ref{convex} satisfy the conditions of Lemma~\ref{begin}, and hence form a
%Koszul relation pair.}
%\end{Example}

\begin{Corollary}
\label{corner}
Let $\MP$ be a convex polyomino, and let $[(1,1),(m,n)]$ be the smallest interval with the property that $V(\MP)\subset [(1,1),(m,n)]$. We assume that $m,n\geq
4$. If one of the vertices $(2,2), (m-1,2), (m-1,n-1)$ or $(2,n-1)$ does not belong to $V(\MP)$, then $I_{\MP}$ has a Koszul relation pair, and, hence, $I_{\MP}$ is
not linearly related.
\end{Corollary}

\begin{proof}
We may assume that $(2,2)\not\in V(\MP)$. Then the vertices of the interval  $[(1,1),(2,2)]$ do not belong to $V(\MP)$. Since $[(1,1),(m,n)]$ is the smallest
interval containing $V(\MP)$, there exist, therefore, integers $i$ and $j$ with $2< i\leq m-1$ and $2< j\leq n-1$ such that the cells $[(i,1),(i+1,2)]$ and
$[(1,j),(2,j+1)]$ belong to $\MP$. Then the collection of cells induced by the rows $1,2,i,i+1$ and the columns $1,2,j,j+1$ is isomorphic to  one of the collections  $\MP'$ of Figure~\ref{restricted}. Thus the assertion follows from Lemma~\ref{begin} and Corollary~\ref{inducedsubpolyomino}.
\end{proof}

Corollary~\ref{corner} shows that  the  convex polyomino $\MP$ should contain all the vertices $(2,2), (m-1,2), (m-1,n-1)$ and $(2,n-1)$ in order to be linearly
related. Thus a polyomino which is linearly related must have the shape as indicated in Figure~\ref{shape}. The number $i_1$ is also allowed to be $1$ in which case also $j_1=1$. In this case the polyomino contains the corner $(1,1)$. A similar  convention applies to the other corners. In Figure~\ref{shape} all for corners $(1,1), (1,n), (m,1)$ and $(m,n)$ are missing.

\begin{figure}[hbt]
\begin{center}
\psset{unit=0.4cm}
\begin{pspicture}(6,-1)(6,12)
{
\pspolygon(0,2)(0,6)(1,6)(1,9)(4,9)(4,10)(10,10)(10,9)(12,9)(12,7)(13,7)(13,3)(12,3)(12,0)(11,0)(11,-1)(4,-1)(4,0)(1,0)(1,2)
}
\rput(-0.5,-0.2){$(2,2)$}
\rput(14.5,-0.2){$(m-1,2)$}
\rput(-1.4,9.2){$(2,n-1)$}
\rput(15.5,9.2){$(m-1,n-1)$}
\rput(1,0){$\bullet$}
\rput(1,9){$\bullet$}
\rput(12,9){$\bullet$}
\rput(12,0){$\bullet$}

\rput(4,-1.6){$i_1$}
\rput(11.2,-1.6){$i_2$}
\rput(4,10.7){$i_3$}
\rput(10.2,10.7){$i_4$}
\rput(-0.7,2.2){$j_1$}
\rput(-0.7,6.1){$j_2$}
\rput(13.7,6.8){$j_4$}
\rput(13.7,3){$j_3$}
\end{pspicture}
\end{center}
\caption{Possible shape}\label{shape}
\end{figure}
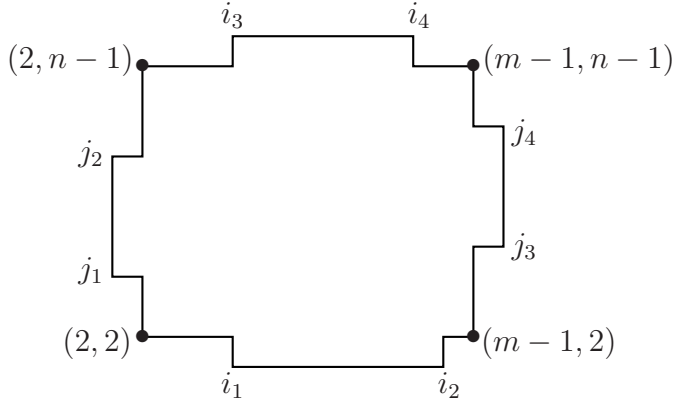

The convex polyomino displayed in Figure~\ref{not} however is not linearly related, though it has the shape as shown in Figure~\ref{shape}. Thus there must still be other obstructions for a polyomino to be
linearly related.
\begin{figure}[hbt]
\begin{center}
\psset{unit=0.9cm}
\begin{pspicture}(4.5,-1)(4.5,3.5)
{
\pspolygon[style=fyp,fillcolor=light](4,0)(4,1)(5,1)(5,0)
\pspolygon[style=fyp,fillcolor=light](5,0)(5,1)(6,1)(6,0)
\pspolygon[style=fyp,fillcolor=light](3,1)(3,2)(4,2)(4,1)
\pspolygon[style=fyp,fillcolor=light](4,1)(4,2)(5,2)(5,1)
\pspolygon[style=fyp,fillcolor=light](3,0)(3,1)(4,1)(4,0)
\pspolygon[style=fyp,fillcolor=light](5,1)(5,2)(6,2)(6,1)
\pspolygon[style=fyp,fillcolor=light](4,2)(4,3)(5,3)(5,2)
\pspolygon[style=fyp,fillcolor=light](5,2)(5,3)(6,3)(6,2)
\pspolygon[style=fyp,fillcolor=light](6,0)(6,1)(7,1)(7,0)
\pspolygon[style=fyp,fillcolor=light](6,1)(6,2)(7,2)(7,1)
\pspolygon[style=fyp,fillcolor=light](4,-1)(4,0)(5,0)(5,-1)
\pspolygon[style=fyp,fillcolor=light](5,-1)(5,0)(6,0)(6,-1)
\pspolygon[style=fyp,fillcolor=light](2,-1)(2,0)(3,0)(3,-1)
\pspolygon[style=fyp,fillcolor=light](3,-1)(3,0)(4,0)(4,-1)
%\pspolygon[style=fyp,fillcolor=light](2,-0)(2,1)(3,1)(3,0)
}
\end{pspicture}
\end{center}
\caption{Not linearly related}\label{not}
\end{figure}
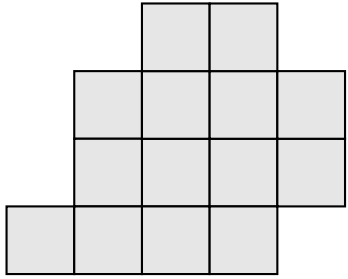

Now we proceed further in eliminating those polyominoes which are not linearly related.

\begin{Lemma}
\label{opposite}
Let $\MP$ be a convex polyomino, and let $[(1,1),(m,n)]$ be the smallest interval with the property that $V(\MP)\subset [(1,1),(m,n)]$. If $\MP$ misses only two  opposite corners, say $(1,1)$ and $(m,n)$, or  $\MP$ misses all four corners $(1,1)$, $(1,n)$, $(m,1)$ and $(m,n)$, then $I_\MP$ admits a Koszul pair and hence is not linearly related.
\end{Lemma}

\begin{proof}
Let us first assume that $(1,1)$ and $(m,n)$ do not belong to $V(\MP)$, but $(1,n)$ and $(m,1)$ belong to $V(\MP)$. %Then in Figure~\ref{shape} we have $i_1>1$,   $i_2=m$ $i_3=1$,  $i_4<m$, $j_1>1$, $j_2=n$, $j_3=1$ and $j_4<n$.
The collection of cells  $\MP_1$ induced by the rows $1,2,m-1,m$ and the columns $1,2,n-1,n$  is  shown in Figure~\ref{stairs}. All the light colored cells, some of them or none of them are present according to whether or not all, some or none of the equations $i_1=2$, $j_1=2$, $i_4=m-1$  and $j_4=n-1$ hold. For example, if $i_1=2$, $j_1\neq 2$, $i_4=m-1$ and $j_4\neq n-1$, then the light colored  cells $[(2,1),(3,2)]$ and $[(2,3),(3,4)]$ belong $\MP_1$ and the other  two light colored cells do not belong to $\MP_1$.

\begin{figure}[hbt]
\begin{center}
\psset{unit=0.9cm}
\begin{pspicture}(4.5,-1)(4.5,3.5)
{
%\pspolygon[style=fyp,fillcolor=light](2.8,0)(2.8,1)(3.8,1)(3.8,0)
\pspolygon[style=fyp,fillcolor=light](4.8,-1)(4.8,0)(5.8,0)(5.8,-1)
%\pspolygon[style=fyp,fillcolor=light](4.8,0)(4.8,1)(5.8,1)(5.8,0)
\pspolygon[style=fyp,fillcolor=light](2.8,1)(2.8,2)(3.8,2)(3.8,1)
\pspolygon[style=fyp,fillcolor=light](3.8,0)(3.8,1)(4.8,1)(4.8,0)
\pspolygon[style=fyp](2.8,0)(2.8,1)(3.8,1)(3.8,0)
\pspolygon[style=fyp](3.8,-1)(3.8,0)(4.8,0)(4.8,-1)
\pspolygon[style=fyp](3.8,1)(3.8,2)(4.8,2)(4.8,1)
\pspolygon[style=fyp](4.8,0)(4.8,1)(5.8,1)(5.8,0)
}
\end{pspicture}
\end{center}
\caption{}\label{stairs}
\end{figure}

It can easily be checked that the ideal  $I_{\MP_1}$  displayed in Figure~\ref{stairs}  has a Koszul relation pairs in all possible cases, and so does $I_\MP$ by Corollary~\ref{inducedsubpolyomino}.

Next, we assume that none of the four corners  $(1,1)$, $(1,n)$, $(m,1)$ and $(m,n)$  belong to  $\MP$. In the following arguments we refer to Figure~\ref{shape}. In the first case suppose $[i_3,i_4]\subset [i_1,i_2]$ and 
$[j_3,j_4]\subset [j_1,j_2]$.  Then the collection of cells induced by the columns $2,i_3,i_4,m-1$  and the rows $1,j_3,j_4,n$ is the polyomino displayed in Figure~\ref{convex} which has a Koszul relation pair as can be verified by computer. Thus $\MP$ has a Koszul relation pair. A similar argument applies if $[i_1,i_2]\subset [i_3,i_4]$ or $[j_1,j_2]\subset [j_3,j_4]$.

Next assume that $[i_3,i_4]\not\subset [i_1,i_2]$ or  $[j_3,j_4]\not\subset [j_1,j_2]$. By symmetry, we may discuss only 
$[i_3,i_4]\not\subset [i_1,i_2]$.
Then we may assume that $i_3<i_1$ and $i_4< i_2$.  We choose the columns $i_1,i_2,i_3,i_4$ and the rows $1,2,n-1,n$. Then the induced polyomino  by these rows and columns is $\MP_1$ if  $i_1<i_4$, $\MP_2$ if  $i_4=i_1$ and $\MP_3$ if  $i_4<i_1$; see Figure~\ref{cases}. In all three cases the corresponding induced polyomino ideal has a Koszul relation pair, and hence so does $I_\MP$.
\end{proof}

\begin{figure}[hbt]
\begin{center}
\psset{unit=0.9cm}
\begin{pspicture}(4.5,-1)(4.5,3.5)
\rput(-6,1){
%\pspolygon[style=fyp,fillcolor=light](2.8,0)(2.8,1)(3.8,1)(3.8,0)
\pspolygon[style=fyp,fillcolor=light](4.8,-1)(4.8,0)(5.8,0)(5.8,-1)
%\pspolygon[style=fyp,fillcolor=light](4.8,0)(4.8,1)(5.8,1)(5.8,0)
\pspolygon[style=fyp,fillcolor=light](2.8,1)(2.8,2)(3.8,2)(3.8,1)
\pspolygon[style=fyp,fillcolor=light](3.8,0)(3.8,1)(4.8,1)(4.8,0)
\pspolygon[style=fyp,fillcolor=light](2.8,0)(2.8,1)(3.8,1)(3.8,0)
\pspolygon[style=fyp,fillcolor=light](3.8,-1)(3.8,0)(4.8,0)(4.8,-1)
\pspolygon[style=fyp, fillcolor=light](3.8,1)(3.8,2)(4.8,2)(4.8,1)
\pspolygon[style=fyp, fillcolor=light](4.8,0)(4.8,1)(5.8,1)(5.8,0)
\rput(4,-2){$i_1<i_4$}
}

\rput(-0.5,1){
%\pspolygon[style=fyp,fillcolor=light](2.8,0)(2.8,1)(3.8,1)(3.8,0)
\pspolygon[style=fyp,fillcolor=light](4.8,-1)(4.8,0)(5.8,0)(5.8,-1)
%\pspolygon[style=fyp,fillcolor=light](4.8,0)(4.8,1)(5.8,1)(5.8,0)
%\pspolygon[style=fyp,fillcolor=light](2.8,1)(2.8,2)(3.8,2)(3.8,1)
\pspolygon[style=fyp,fillcolor=light](3.8,0)(3.8,1)(4.8,1)(4.8,0)
%\pspolygon[style=fyp,fillcolor=light](2.8,0)(2.8,1)(3.8,1)(3.8,0)
%\pspolygon[style=fyp,fillcolor=light](3.8,-1)(3.8,0)(4.8,0)(4.8,-1)
\pspolygon[style=fyp, fillcolor=light](3.8,1)(3.8,2)(4.8,2)(4.8,1)
\pspolygon[style=fyp, fillcolor=light](4.8,0)(4.8,1)(5.8,1)(5.8,0)
\rput(5,-2){$i_4=i_1$}
}

\rput(5.5,1){
\pspolygon[style=fyp,fillcolor=light](2.8,0)(2.8,1)(3.8,1)(3.8,0)
\pspolygon[style=fyp,fillcolor=light](4.8,-1)(4.8,0)(5.8,0)(5.8,-1)
%\pspolygon[style=fyp,fillcolor=light](4.8,0)(4.8,1)(5.8,1)(5.8,0)
\pspolygon[style=fyp,fillcolor=light](2.8,1)(2.8,2)(3.8,2)(3.8,1)
\pspolygon[style=fyp,fillcolor=light](3.8,0)(3.8,1)(4.8,1)(4.8,0)
\pspolygon[style=fyp,fillcolor=light](2.8,0)(2.8,1)(3.8,1)(3.8,0)
%\pspolygon[style=fyp,fillcolor=light](3.8,-1)(3.8,0)(4.8,0)(4.8,-1)
%\pspolygon[style=fyp, fillcolor=light](3.8,1)(3.8,2)(4.8,2)(4.8,1)
\pspolygon[style=fyp, fillcolor=light](4.8,0)(4.8,1)(5.8,1)(5.8,0)
\rput(4.5,-2){$i_4<i_1$}
}
\end{pspicture}
\end{center}
\caption{}\label{cases}
\end{figure}

\begin{Lemma}
\label{threecorners}
Let $\MP$ be a convex polyomino, and let $[(1,1),(m,n)]$ be the smallest interval with the property that $V(\MP)\subset [(1,1),(m,n)]$. Suppose $\MP$ misses three corners, say $(1,n),(m,1),(m,n)$, and suppose that  $i_2<m-1$ and  $j_2<n-1$, or $i_2=m-1$ and $j_2<j_4$, or  $j_2=n-1$ and $i_2< i_4$. Then $I_\MP$ has a Koszul relation pair and hence is not linearly related.
\end{Lemma}

\begin{proof}
We proceed as in the proofs of the previous lemmata. In the case that $i_2<m-1$ and  $j_2<n-1$, we consider the collection of cells $\MP'$ induced by the columns $1,2,m-1$ and the rows $1,2,n-1$.  This collection of cells $\MP'$  is depicted in Figure~\ref{two}. It is easily seen that $I_{\MP'}$ is generated by a regular sequence of length $2$, which is a Koszul relation pair. In the case that $i_2=m-1$ and $j_2< j_4$ we choose the columns $1,2,m-1,m$ and the rows $1,2,j_4-1,j_4$. The polyomino $\MP''$ induced by this choice of rows and columns has two opposite missing corners, hence, by Lemma~\ref{opposite}, it has a Koszul pair. The case  $j_2=n-1$ and $i_2< i_4$ is symmetric. In both cases the induced polyomino ideal has a Koszul relation pair. Hence in all three cases $I_\MP$ itself has a Koszul relation pair.
\begin{figure}[hbt]
\begin{center}
\psset{unit=0.9cm}
\begin{pspicture}(3.5,-1)(3.5,1)
\pspolygon[style=fyp,fillcolor=light](2.8,-1)(2.8,0)(3.8,0)(3.8,-1)
%\pspolygon[style=fyp,fillcolor=light](2.8,0)(2.8,1)(3.8,1)(3.8,0)
%\pspolygon[style=fyp,fillcolor=light](3.8,-1)(3.8,0)(4.8,0)(4.8,-1)
%\pspolygon[style=fyp,fillcolor=light](4.8,0)(4.8,1)(5.8,1)(5.8,0)
%\pspolygon[style=fyp,fillcolor=light](4.8,1)(4.8,2)(5.8,2)(5.8,1)
%\pspolygon[style=fyp,fillcolor=light](3.8,1)(3.8,2)(4.8,2)(4.8,1)
\pspolygon[style=fyp,fillcolor=light](3.8,0)(3.8,1)(4.8,1)(4.8,0)

%\rput(8.3,-1){$\MP'$}
\end{pspicture}
\end{center}
\caption{}\label{two}
\end{figure}
\end{proof}

\begin{proof}[Proof of Theorem~\ref{main}]
Implication (a)$\Rightarrow$(b) is obvious. Implication (b)$\Rightarrow$(c) follows by Corollary~\ref{corner}, Lemma~\ref{opposite},
and Lemma~\ref{threecorners}.

It remains to  prove (c)$\Rightarrow$(a). Let $\MP$ be a convex polyomino which satisfies one of the conditions (i)--(iii).
We have to show that $\MP$ is linearly related. By Corollary~\ref{no}, we only need to prove that $\beta_{14}(I_{\MP})=0.$
Viewing $K[\MP]$ as a semigroup ring $K[H],$ it follows that one has to check that $\beta_{1h}(I_{\MP})=0$ for all
$h\in H$ with $|h|=4.$ The main idea of this proof is to use Corollary~\ref{refinement}.

Let $h=h_1h_2h_3h_4$ with $j_q=s_{i_q}t_{i_q}$ for $1\leq q\leq 4,$ and
$
i=\min_{q}\{i_q\}, k=\max_{q}\{i_q\}, j=\min_{q}\{j_q\}, \text{ and } \ell=\max_{q}\{j_q\}.
$
Therefore, all the points $h_q$ lie in the (possible degenerate) rectangle $\MQ$ of vertices
$(i,j), (k,j), (i,\ell),(k,\ell).$ If $\MQ$ is degenerate, that is, all the vertices of $Q$ are contained in a vertical or horizontal line segment in $\MP$, then $\beta_{1h}(I_{\MP})=0$ since in this case the simplicial complex $\Delta_h$ is just a simplex. Let us now consider $\MQ$ non-degenerate. If all the vertices of $\MQ$ belong to $\MP$, then the rectangle $\MQ$ is an induced subpolyomino of $\MP$. Therefore, by Corollary~\ref{refinement}, we have
$\beta_{1h}(I_{\MP})=\beta_{1h}(I_Q)=0$, the latter equality being true since $\MQ$ is linearly related.

Next, let us assume that some of the vertices of $\MQ$ do not belong to $\MP.$ As $\MP$ has one of the forms (i)--(iii), it follows that at most three verices of $\MQ$ do not belong to $\MP.$ Consequently, we have to analyze the following cases.

{\em Case 1.} Exactly one vertex of $\MQ$ does not belong to $\MP.$ Without loss of generality, we may assume that $(k,\ell)\notin \MP$ which implies that $k=m$ and $\ell=n.$ In this case, any relation in degree $h$ of $\MP$ is a relation of same degree of one of the polyominoes displayed in Figure~\ref{proof1corner}.

\begin{figure}[hbt]
\begin{center}
\psset{unit=0.9cm}
\begin{pspicture}(4.5,-6)(4.5,3.5)

\rput(-4,0.5){
\pspolygon[style=fyp,fillcolor=light](2.8,-1)(2.8,0)(3.8,0)(3.8,-1)
\pspolygon[style=fyp,fillcolor=light](2.8,0)(2.8,1)(3.8,1)(3.8,0)
\pspolygon[style=fyp,fillcolor=light](3.8,-1)(3.8,0)(4.8,0)(4.8,-1)
\pspolygon[style=fyp,fillcolor=light](5.8,-1)(5.8,0)(4.8,0)(4.8,-1)
\pspolygon[style=fyp,fillcolor=light](4.8,0)(4.8,1)(5.8,1)(5.8,0)
\pspolygon[style=fyp,fillcolor=light](2.8,1)(2.8,2)(3.8,2)(3.8,1)
\pspolygon[style=fyp,fillcolor=light](3.8,1)(3.8,2)(4.8,2)(4.8,1)
\pspolygon[style=fyp,fillcolor=light](3.8,0)(3.8,1)(4.8,1)(4.8,0)
}
\rput(0.3,-1){(a)}

\rput(4.5,0.5){
\pspolygon[style=fyp,fillcolor=light](2.8,-1)(2.8,0)(3.8,0)(3.8,-1)
\pspolygon[style=fyp,fillcolor=light](2.8,0)(2.8,1)(3.8,1)(3.8,0)
\pspolygon[style=fyp,fillcolor=light](3.8,-1)(3.8,0)(4.8,0)(4.8,-1)
\pspolygon[style=fyp,fillcolor=light](5.8,-1)(5.8,0)(4.8,0)(4.8,-1)
\pspolygon[style=fyp,fillcolor=light](4.8,0)(4.8,1)(5.8,1)(5.8,0)
\pspolygon[style=fyp,fillcolor=light](2.8,1)(2.8,2)(3.8,2)(3.8,1)
%\pspolygon[style=fyp,fillcolor=light](3.8,1)(3.8,2)(4.8,2)(4.8,1)
\pspolygon[style=fyp,fillcolor=light](3.8,0)(3.8,1)(4.8,1)(4.8,0)
}
\rput(8.3,-1){(b)}

\rput(-4,-4.5){
\pspolygon[style=fyp,fillcolor=light](2.8,-1)(2.8,0)(3.8,0)(3.8,-1)
\pspolygon[style=fyp,fillcolor=light](2.8,0)(2.8,1)(3.8,1)(3.8,0)
\pspolygon[style=fyp,fillcolor=light](3.8,-1)(3.8,0)(4.8,0)(4.8,-1)
\pspolygon[style=fyp,fillcolor=light](5.8,-1)(5.8,0)(4.8,0)(4.8,-1)
%\pspolygon[style=fyp,fillcolor=light](4.8,0)(4.8,1)(5.8,1)(5.8,0)
\pspolygon[style=fyp,fillcolor=light](2.8,1)(2.8,2)(3.8,2)(3.8,1)
\pspolygon[style=fyp,fillcolor=light](3.8,1)(3.8,2)(4.8,2)(4.8,1)
\pspolygon[style=fyp,fillcolor=light](3.8,0)(3.8,1)(4.8,1)(4.8,0)
}
\rput(0.3,-6){(c)}

\rput(4.5,-4.5){
\pspolygon[style=fyp,fillcolor=light](2.8,-1)(2.8,0)(3.8,0)(3.8,-1)
\pspolygon[style=fyp,fillcolor=light](2.8,0)(2.8,1)(3.8,1)(3.8,0)
\pspolygon[style=fyp,fillcolor=light](3.8,-1)(3.8,0)(4.8,0)(4.8,-1)
\pspolygon[style=fyp,fillcolor=light](5.8,-1)(5.8,0)(4.8,0)(4.8,-1)
%\pspolygon[style=fyp,fillcolor=light](4.8,0)(4.8,1)(5.8,1)(5.8,0)
\pspolygon[style=fyp,fillcolor=light](2.8,1)(2.8,2)(3.8,2)(3.8,1)
%\pspolygon[style=fyp,fillcolor=light](3.8,1)(3.8,2)(4.8,2)(4.8,1)
\pspolygon[style=fyp,fillcolor=light](3.8,0)(3.8,1)(4.8,1)(4.8,0)
}
\rput(8.3,-6){(d)}
\end{pspicture}
\end{center}
\caption{}\label{proof1corner}
\end{figure}
One may check with a computer algebra system that all polyominoes displayed in Figure~\ref{proof1corner} are linearly related, hence they do not have any relation in degree $h.$ Actually, one has to check only the shapes  (a), (b), and
(d) since the polyomino displayed in (c) is isomorphic to that one from (b). Hence, $\beta_{1h}(I_{\MP})=0.$

{\em Case 2.} Two vertices of $\MQ$ do not belong to $\MP.$ We may assume that the missing vertices from $\MP$ are
$(i,\ell)$and $(k,\ell)$. Hence, we have $i=1,$ $k=m$, and $\ell=n.$ In this case, any relation in degree $h$ of $\MP$ is a relation of same degree of one of the polyominoes displayed in Figure~\ref{proof23corner} (a)--(c). Note that the polyominoes  (b) and (c) are isomorphic. One easily checks with the computer that all these polyominoes are linearly related, thus  $\beta_{1h}(I_{\MP})=0.$

\begin{figure}[hbt]
\begin{center}
\psset{unit=0.9cm}
\begin{pspicture}(4.5,-6)(4.5,3.5)

\rput(-4,0.5){
\pspolygon[style=fyp,fillcolor=light](2.8,-1)(2.8,0)(3.8,0)(3.8,-1)
\pspolygon[style=fyp,fillcolor=light](2.8,0)(2.8,1)(3.8,1)(3.8,0)
\pspolygon[style=fyp,fillcolor=light](3.8,-1)(3.8,0)(4.8,0)(4.8,-1)
\pspolygon[style=fyp,fillcolor=light](5.8,-1)(5.8,0)(4.8,0)(4.8,-1)
\pspolygon[style=fyp,fillcolor=light](4.8,0)(4.8,1)(5.8,1)(5.8,0)
%\pspolygon[style=fyp,fillcolor=light](2.8,1)(2.8,2)(3.8,2)(3.8,1)
\pspolygon[style=fyp,fillcolor=light](3.8,1)(3.8,2)(4.8,2)(4.8,1)
\pspolygon[style=fyp,fillcolor=light](3.8,0)(3.8,1)(4.8,1)(4.8,0)
}
\rput(0.3,-1){(a)}

\rput(4.5,0.5){
\pspolygon[style=fyp,fillcolor=light](2.8,-1)(2.8,0)(3.8,0)(3.8,-1)
%\pspolygon[style=fyp,fillcolor=light](2.8,0)(2.8,1)(3.8,1)(3.8,0)
\pspolygon[style=fyp,fillcolor=light](3.8,-1)(3.8,0)(4.8,0)(4.8,-1)
\pspolygon[style=fyp,fillcolor=light](5.8,-1)(5.8,0)(4.8,0)(4.8,-1)
\pspolygon[style=fyp,fillcolor=light](4.8,0)(4.8,1)(5.8,1)(5.8,0)
%\pspolygon[style=fyp,fillcolor=light](2.8,1)(2.8,2)(3.8,2)(3.8,1)
\pspolygon[style=fyp,fillcolor=light](3.8,1)(3.8,2)(4.8,2)(4.8,1)
\pspolygon[style=fyp,fillcolor=light](3.8,0)(3.8,1)(4.8,1)(4.8,0)
}
\rput(8.3,-1){(b)}

\rput(-4,-4.5){
\pspolygon[style=fyp,fillcolor=light](2.8,-1)(2.8,0)(3.8,0)(3.8,-1)
\pspolygon[style=fyp,fillcolor=light](2.8,0)(2.8,1)(3.8,1)(3.8,0)
\pspolygon[style=fyp,fillcolor=light](3.8,-1)(3.8,0)(4.8,0)(4.8,-1)
\pspolygon[style=fyp,fillcolor=light](5.8,-1)(5.8,0)(4.8,0)(4.8,-1)
%\pspolygon[style=fyp,fillcolor=light](4.8,0)(4.8,1)(5.8,1)(5.8,0)
%\pspolygon[style=fyp,fillcolor=light](2.8,1)(2.8,2)(3.8,2)(3.8,1)
\pspolygon[style=fyp,fillcolor=light](3.8,1)(3.8,2)(4.8,2)(4.8,1)
\pspolygon[style=fyp,fillcolor=light](3.8,0)(3.8,1)(4.8,1)(4.8,0)
}
\rput(0.3,-6){(c)}

\rput(4.5,-4.5){
\pspolygon[style=fyp,fillcolor=light](2.8,-1)(2.8,0)(3.8,0)(3.8,-1)
\pspolygon[style=fyp,fillcolor=light](2.8,0)(2.8,1)(3.8,1)(3.8,0)
\pspolygon[style=fyp,fillcolor=light](3.8,-1)(3.8,0)(4.8,0)(4.8,-1)
%\pspolygon[style=fyp,fillcolor=light](5.8,-1)(5.8,0)(4.8,0)(4.8,-1)
\pspolygon[style=fyp,fillcolor=light](4.8,0)(4.8,1)(5.8,1)(5.8,0)
%\pspolygon[style=fyp,fillcolor=light](2.8,1)(2.8,2)(3.8,2)(3.8,1)
\pspolygon[style=fyp,fillcolor=light](3.8,1)(3.8,2)(4.8,2)(4.8,1)
\pspolygon[style=fyp,fillcolor=light](3.8,0)(3.8,1)(4.8,1)(4.8,0)
}
\rput(8.3,-6){(d)}
\end{pspicture}
\end{center}
\caption{}\label{proof23corner}
\end{figure}

{\em Case 3.} Finally, we assume that there are three vertices of $\MQ$ which do not belong to $\MP.$ We may assume that these vertices are $(i,\ell), (k,\ell),$ and $(k,j)$. In this case, any relation in degree $h$ of $\MP$ is a relation of same degree of the polyomino displayed in Figure~\ref{proof23corner} (d) which is linearly related as one may easily check with the computer. Therefore, we get again $\beta_{1h}(I_{\MP})=0.$
\end{proof}

\section{Polyomino ideals with linear resolution}

In this  final section, we classify all convex polyominoes which have a linear resolution and the convex stack polyominoes which are extremal Gorenstein.

\begin{Theorem}
\label{linear}
Let $\MP$ be a convex polyomino. Then the following conditions are equivalent:
\begin{enumerate}
\item[\em (a)] $I_\MP$ has a linear resolution;
\item[\em (b)] there exists a positive  integer $m$ such that $\MP$ is isomorphic to the polyomino with cells $[(i,i), (i+1,i+1)]$, $i=1,\ldots,m-1$.
\end{enumerate}
\end{Theorem}

\begin{proof}
(b)\implies (a): If the polyomino is of the shape as described in (b), then $I_\MP$ is just the ideal of $2$-minors of a $2\times m$-matrix. It is well-known that the ideal of $2$-minors of such a matrix has a linear resolution. Indeed the Eagon-Northcott complex, whose chain maps are described by matrices with linear entries,   provides a free resolution of the ideal of maximal minors of any matrix of indeterminates,   see for example \cite[Page 600]{Ei}.

(a)\implies (b): We may assume   that $[(1,1),(m,n)]$  is the smallest interval containing $V(\MP)$.  We may further assume that $m\geq 4$ or $n\geq 4$. The few remaining cases can easily be checked with the computer. So let us assume that
$m\geq 4$. Then we have to show that $n=2$. Suppose that $n\geq 3$.  We first assume   that all the corners $(1,1),(1,n),(m,1)$ and $(m,n)$ belong to $V(\MP)$. Then the polyomino $\MP'$ induced by the columns  $1,2,m$ and the rows  $1,2,n$ is the polyomino which is displayed on the right of Figure~\ref{extremalstack}. The ideal $I_{\MP'}$ is a Gorenstein ideal, and hence it is does not have a linear resolution. Therefore, by Corollary~\ref{inducedsubpolyomino}, the ideal $I_\MP$ does not have a linear resolution as well, a contradiction.

Next assume that one of the corners, say $(1,1)$, is missing. Since $I_\MP$ has a linear a linear resolution, $I_\MP$ is linearly related and hence has a shape as indicated in Figure~\ref{shape}. Let $i_1$ and $j_1$ be the numbers as shown in Figure~\ref{shape}, and let $\MP'$ the polyomino of $\MP$ induced by the columns  $1,2,3$ and the rows $a,j_1,j_1+1$ where $a=1$ if $i_1=2$ and $a=2$ if $i_1>2, j_1>2$. If $j_1=2$ and $i_1>2,$ we let $\MP^\prime$ to be the polyomino induced by the columns $1,i_1,i_1+1$ and the rows $1,2,3.$ In any case, $\MP'$ is isomorphic to that one displayed on the left of Figure~\ref{extremalstack}. Since $I_{\MP'}$ is again a Gorenstein ideal, we conclude, as in the first case, that $I_\MP$ does not a have linear resolution, a contradiction.
\end{proof}

As mentioned in the introduction, polyomino ideals overlap with  join-meet  ideals of planar lattices. In the next result we show that the join-meet ideal of any
lattice has linear resolution if and only if it is a polyomino as described in Theorem~\ref{linear}. With methods different from those which are used in this
paper, the classification of join-meet  ideals with linear resolution  was first given  in \cite[Corollary 10]{ERQ}.

Let $L$ be a finite distributive lattice
%
% takayuki 09.03.14
%
\cite[pp.~118]{hibiredsbook}.
%
% takayuki 09.03.14
%
A {\em join-irreducible} element of $L$ is an element $\alpha \in L$
which is not a unique minimal element and which possesses the property that
$\alpha \neq \beta \vee \gamma$
for all $\beta, \, \gamma \in L \setminus \{\alpha\}$.
Let $P$ be the set of join-irreducible elements of $L$.
We regard $P$ as a poset (partially ordered set)
which inherits its ordering from that of $L$.
A subset $J$ of $P$ is called an {\em order ideal} of $P$
if $a \in J$, $b \in P$ together with $b \leq a$ imply $b \in J$.
In particular, the empty set of $P$ is an order ideal of $P$.
Let ${\mathcal J}(P)$ denote the set of order ideals of $P$,
ordered by inclusion.  It then follows that ${\mathcal J}(P)$ is a distributive
lattice.  Moreover, Birkhoff's fundamental structure theorem of finite
distributive lattices
%
% takayuki 09.03.14
%
\cite[Proposition 37.13]{hibiredsbook} 
%
% takayuki 09.03.14
%
guarantees that $L$  coincides with ${\mathcal J}(P)$.

Let $L = {\mathcal J}(P)$ be a finite distributive lattice
and $K[L] = K[ \, x_{\alpha} : \alpha \in L \, ]$ the polynomial ring
in $|L|$ variables over $K$.  The {\em join-meet ideal} $I_{L}$
of $L$ is the ideal of $K[L]$ which is generated by those binomials
\[
x_{\alpha}x_{\beta} - x_{\alpha \wedge \beta}x_{\alpha \vee \beta},
\]
where $\alpha, \, \beta \in L$ are incomparable in $L$.
It is known \cite{Hibi} that $I_{L}$ is a prime ideal and the quotient
ring $K[L]/I_{L}$ is normal and Cohen--Macaulay.
Moreover, $K[L]/I_{L}$ is Gorenstein if and only if $P$ is pure.
(A finite poset is {\em pure} if every maximal chain
(totally ordered subset) of $P$ has the same cardinality.)

Now, let $P = \{ \xi_{1}, \ldots, \xi_{d} \}$ be a finite poset, where
$i < j$ if $\xi_{i} < \xi_{j}$, and $L = {\mathcal J}(P)$.
A {\em linear extension} of $P$ is a permutation $\pi = i_{1} \cdots i_{d}$
of $[n] = \{ 1, \ldots, n \}$ such that $j < j'$ if $\xi_{i_{j}} < \xi_{i_{j'}}$.
A {\em descent} of $\pi = i_{1} \cdots i_{d}$ is an index $j$ with
$i_{j} > i_{j+1}$.  Let $D(\pi)$ denote the set of descents of $\pi$.
The {\em $h$-vector} of $L$ is the sequence
$h(L) = (h_{0}, h_{1}, \ldots, h_{d-1})$,
where $h_{i}$ is the number of permutations $\pi$
of $[n]$ with $|D(\pi)| = i$.
Thus, in particular, $h_{0} = 1$.
It follows from
%
% takayuki 09.03.14
%
\cite{BGS}
%
% takayuki 09.03.14
%
that the Hilbert series of $K[L]/I_{L}$
is of the form
\[
\frac{h_{0} + h_{1} \lambda + \cdots + h_{d-1} \lambda^{d-1}}{(1 - \lambda)^{d+1}}.
\]

We say that a finite distributive lattice $L = {\mathcal J}(P)$ is {\em simple}
if $L$ has no elements $\alpha$ and $\beta$ with $\beta < \alpha$ such that
each element $\gamma \in L \setminus \{\alpha, \beta\}$
satisfies either $\gamma < \beta$ or $\gamma > \alpha$.
In other words, $L$ is simple if and only if $P$ possesses no element
$\xi$ for which every $\mu \in P$ satisfies either $\mu \leq \xi$ or
$\mu \geq \xi$.

\begin{Theorem}
\label{hibione}
Let $L = {\mathcal J}(P)$ be a simple finite distributive lattice.
Then the join-meet ideal $I_{L}$ has a linear resolution
if and only if $L$ is of the form shown in Figure~\ref{plane}.

\begin{figure}[bht]
\begin{center}
\psset{unit=0.4cm}
\begin{pspicture}(-10.3,-2)(4,9)

%\rput(-3,0){
%\psline(-9,-0.78)(-9,0.28)
%\psline(-9,0.72)(-9,1.78)
%\psline(-9,2.23)(-9,3.4)
%\rput(-9,4.2){$\vdots$}
%\psline(-9,4.5)(-9,5.83)
%\psline(-9,6.23)(-9,7.45)
%\rput(-9,6){$\bullet$}
%\rput(-9,7.6){$\bullet$}
%\rput(-9,2){$\bullet$}
%\rput(-9,0.5){$\bullet$}
%\rput(-9,-1){$\bullet$}
%\rput(-11,-1){$\bullet$}
%}

\rput(1,0){
\pspolygon(-11,1)(-9,-1)(-7,1)(-9,3)
\pspolygon(-7,1)(-5,3)(-7,5)(-9,3)
\psline[linestyle=dotted](-5,3)(-3,5)
\psline[linestyle=dotted](-7,5)(-5,7)
\pspolygon(-3,5)(-5,7)(-3,9)(-1,7)
\rput(-11,1){$\bullet$}
\rput(-9,-1){$\bullet$}
\rput(-7,1){$\bullet$}
\rput(-9,3){$\bullet$}
\rput(-5,3){$\bullet$}
\rput(-7,5){$\bullet$}
\rput(-3,5){$\bullet$}
\rput(-5,7){$\bullet$}
\rput(-3,9){$\bullet$}
\rput(-1,7){$\bullet$}
}

\end{pspicture}
\end{center}
\caption{}\label{plane}
\end{figure}

\end{Theorem}

\begin{proof}
Since $I_L$ is generated in degree $2$, it follows that $I_L$  has a linear resolution if and only if the regularity of $K[L]/I_{L}$  is equal to $1$. We may assume that $K$ is infinite.  Since $K[L]/I_{L}$ is Cohen--Macaulay, we may divide by a regular sequence of linear forms to obtain a $0$-dimensional $K$-algebra $A$ with $\reg A=\reg K[L]/I_{L}$ whose $h$-vector coincides with that of $\reg K[L]/I_{L}$. Since $\reg A=\max\{i\: A_i\neq 0\}$ (see for example \cite[Exercise 20.18]{Ei}), it follows that
 $I_{L}$ has a linear
resolution if and only if the $h$-vector of $L$ is of the form
$h(L) = (1, q, 0, \ldots, 0)$,
where $q \geq 0$ is an integer.
Clearly, if $P$ is a finite poset of Figure~\ref{plane}, then $|D(\pi)| \leq 1$
for each linear extension $\pi$ of $P$.  Thus $I_{L}$ has a linear resolution.

Conversely, suppose that $I_{L}$ has a linear resolution.  In other words,
one has $|D(\pi)| \leq 1$ for each linear extension $\pi$ of $P$.
Then $P$ has no three-element clutter.
(A {\em clutter} of $P$ is a subset $A$ of $P$ with the property that no two elements belonging to $A$
are comparable in $P$.)  Since $L = {\mathcal J}(P)$ is simple, it follows that
$P$ contains a two-element clutter.  Hence Dilworth's theorem
%
% takayuki 09.03.14
%
\cite{Dil}
%
% takayuki 09.03.14
%
says that $P = C \cup C'$, where $C$ and $C'$ are chains of $P$
with $C \cap C' = \emptyset$.  Let $|C| \geq 2$ and $|C'| \geq 2$.
Let $\xi \in C$ and $\mu \in C'$ be minimal elements of $P$.
Let $\xi' \in C$ and $\mu' \in C'$ be maximal elements of $P$.
Since $L = {\mathcal J}(P)$ is simple,
it follows that $\xi \neq \mu$ and $\xi' \neq \mu'$.
Thus there is a linear extension $\pi$ of $P$ with $|D(\pi)| \geq 2$.
Thus $I_{L}$ cannot have a linear resolution.  Hence either
$|C| = 1$ or $|C'| = 1$, as desired.
\end{proof}

A Gorenstein ideal can never have a linear resolution, unless it is a principal ideal. However, if the resolution is as much linear as possible, then it is called
extremal Gorenstein. Since polyomino ideals are generated in degree $2$ we restrict ourselves in the following definition of extremal Gorenstein ideals to graded ideals generated in degree
$2$.

Let $S$ be a polynomial ring over field, and $I\subset S$ a graded ideal which is not principal and is generated in degree $2$.
Following \cite{Sch} we say that $I$ is
an {\em extremal Gorenstein ideal} if $S/I$ is Gorenstein and
if the shifts of the graded minimal free resolution are
\[
-2-p-1, -2-(p-1), -2-(p-2), \ldots, - 3, -2,
\]
where $p$ is the projective dimension of $I$.

With similar arguments as in the proof of Theorem~\ref{hibione}, we see that $I$
is an extremal Gorenstein ideal if and only if $I$ is a Gorenstein ideal and $\reg S/I=2$, and that this is the case
if and only if $S/I$ is Cohen--Macaulay and the $h$-vector of $S/I$ is of the form
\[
h(L) = (1, q, 1, 0, \ldots, 0),
\]
where $q > 1$ is an integer.

\medskip
In the following theorem we classify all convex stack polyominoes $\MP$ for which $I_\MP$ is extremal Gorenstein. Convex stack polyominoes have been considered in
\cite{Q}. In that paper Qureshi characterizes those convex stack polyominoes $\MP$ for which $I_\MP$ is Gorenstein.

Let $\MP$ be a polyomino.  We may assume that  $[(1,1),(m,n)]$ is the smallest
interval containing $V(\MP)$. Then $\MP$ is called a {\em stack polyomino} if it is column convex and for $i=1,.\ldots,m-1$ the cells $[(i,1),(i+1,2)]$ belong to
$\MP$. Figure~\ref{stackpolyomino} displays stack polyominoes -- the right polyomino is convex, the left is not. The number of  cells of the bottom row  is called the {\em width} of $\MP$ and the  number of cells in a maximal column is called the {\em height} of $\MP$.

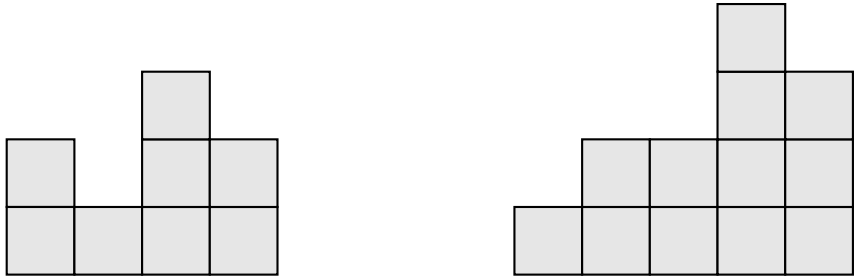
\begin{figure}[hbt]
\begin{center}
\psset{unit=0.9cm}
\begin{pspicture}(5,-0.5)(5,5)
\rput(-5,0){
\pspolygon[style=fyp,fillcolor=light](4,2)(4,3)(5,3)(5,2)
%\pspolygon[style=fyp,fillcolor=light](4,0)(4,1)(5,1)(5,0)
%\pspolygon[style=fyp,fillcolor=light](5,0)(5,1)(6,1)(6,0)
%\pspolygon[style=fyp,fillcolor=light](6,0)(6,1)(7,1)(7,0)
%\pspolygon[style=fyp,fillcolor=light](7,0)(7,1)(8,1)(8,0)
\pspolygon[style=fyp,fillcolor=light](4,1)(4,2)(5,2)(5,1)
\pspolygon[style=fyp,fillcolor=light](5,1)(5,2)(6,2)(6,1)
\pspolygon[style=fyp,fillcolor=light](6,1)(6,2)(7,2)(7,1)
\pspolygon[style=fyp,fillcolor=light](7,1)(7,2)(8,2)(8,1)
\pspolygon[style=fyp,fillcolor=light](6,2)(6,3)(7,3)(7,2)
\pspolygon[style=fyp,fillcolor=light](7,2)(7,3)(8,3)(8,2)
\pspolygon[style=fyp,fillcolor=light](6,3)(6,4)(7,4)(7,3)
}
\rput(3.5,1){
%\pspolygon[style=fyp,fillcolor=light](4,2)(4,3)(5,3)(5,2)
\pspolygon[style=fyp,fillcolor=light](3,0)(3,1)(4,1)(4,0)
\pspolygon[style=fyp,fillcolor=light](4,0)(4,1)(5,1)(5,0)
\pspolygon[style=fyp,fillcolor=light](5,0)(5,1)(6,1)(6,0)
\pspolygon[style=fyp,fillcolor=light](6,0)(6,1)(7,1)(7,0)
\pspolygon[style=fyp,fillcolor=light](7,0)(7,1)(8,1)(8,0)
\pspolygon[style=fyp,fillcolor=light](4,1)(4,2)(5,2)(5,1)
\pspolygon[style=fyp,fillcolor=light](5,1)(5,2)(6,2)(6,1)
\pspolygon[style=fyp,fillcolor=light](6,1)(6,2)(7,2)(7,1)
\pspolygon[style=fyp,fillcolor=light](7,1)(7,2)(8,2)(8,1)
\pspolygon[style=fyp,fillcolor=light](6,2)(6,3)(7,3)(7,2)
\pspolygon[style=fyp,fillcolor=light](7,2)(7,3)(8,3)(8,2)
\pspolygon[style=fyp,fillcolor=light](6,3)(6,4)(7,4)(7,3)
}
\end{pspicture}
\end{center}
\caption{Stack polyominoes}\label{stackpolyomino}
\end{figure}

Let $\MP$ be a convex stack polyomino. Removing the first $k$ bottom rows of cells of $\MP$ we obtain again a convex stack polyomino which we denote by $\MP_k$. We also set $\MP_0=\MP$. Let $h$ be the height of the polymino, and let   $1<k_1<k_2 < \cdots <k_r <h$ be the numbers with the property that $\width(\MP_{k_i})<\width(\MP_{k_{i-1}})$. Furthermore, we set $k_0=1$. For example, for the convex stack polyomino in Figure~\ref{stackpolyomino} we have $k_1=1$, $k_2=2$ and $k_3=3$.

With the terminology and notation introduced,  the characterization of Gorenstein convex stack polyominoes is given in the following theorem.

\begin{Theorem}[Qureshi]
\label{gorensteinstack}
Let $\MP$ be a convex stack polyomino of height $h$. Then the following conditions are equivalent:
\begin{enumerate}
\item[{\em (a)}] $I_\MP$ is a Gorenstein ideal.
\item[{\em (b)}] $\width(\MP_{k_i})=\height(\MP_{k_i})$ for $i=0,\ldots,r$.
\end{enumerate}
\end{Theorem}

According to this theorem, the convex stack polyomino displayed in Figure~\ref{stackpolyomino} is not Gorenstein, because $\width(\MP_{k_0})=5$ and $\height(\MP_{k_0})=4$. An example of a Gorenstein stack polyomino is shown in Figure~\ref{gorenstein}.

\begin{figure}[hbt]
\begin{center}
\psset{unit=0.9cm}
\begin{pspicture}(0.5,-0.5)(0.5,4)
\rput(-6.5,0){
%\pspolygon[style=fyp,fillcolor=light](4,0)(4,1)(5,1)(5,0)
\pspolygon[style=fyp,fillcolor=light](5,0)(5,1)(6,1)(6,0)
\pspolygon[style=fyp,fillcolor=light](6,0)(6,1)(7,1)(7,0)
\pspolygon[style=fyp,fillcolor=light](7,0)(7,1)(8,1)(8,0)
\pspolygon[style=fyp,fillcolor=light](8,0)(8,1)(9,1)(9,0)
%\pspolygon[style=fyp,fillcolor=light](8,1)(8,2)(9,2)(9,1)
\pspolygon[style=fyp,fillcolor=light](5,1)(5,2)(6,2)(6,1)
\pspolygon[style=fyp,fillcolor=light](6,1)(6,2)(7,2)(7,1)
\pspolygon[style=fyp,fillcolor=light](7,1)(7,2)(8,2)(8,1)
\pspolygon[style=fyp,fillcolor=light](5,2)(5,3)(6,3)(6,2)
\pspolygon[style=fyp,fillcolor=light](6,2)(6,3)(7,3)(7,2)
%\pspolygon[style=fyp,fillcolor=light](7,2)(7,3)(8,3)(8,2)
\pspolygon[style=fyp,fillcolor=light](6,3)(6,4)(7,4)(7,3)
}
\end{pspicture}
\end{center}
\caption{A Gorenstein stack polyomino}\label{gorenstein}
\end{figure}
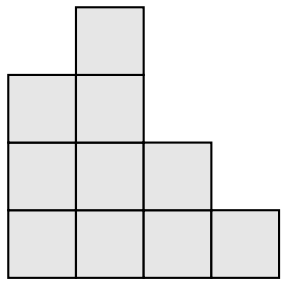

Combining Theorem~\ref{gorensteinstack} with the results of Section 2, we obtain

\begin{Theorem}
\label{stack}
Let $I_\MP$ be  convex stack polyomino. Then $I_\MP$ is extremal Gorenstein if and only if $\MP$ is isomorphic to one of the polyominoes in
Figure~\ref{extremalstack}.

\begin{figure}[hbt]
\begin{center}
\psset{unit=0.9cm}
\begin{pspicture}(4.5,-1)(4.5,3.5)
{
\pspolygon[style=fyp,fillcolor=light](0.8,-1)(0.8,0)(1.8,0)(1.8,-1)
\pspolygon[style=fyp,fillcolor=light](1.8,0)(1.8,1)(2.8,1)(2.8,0)
\pspolygon[style=fyp,fillcolor=light](1.8,-1)(1.8,0)(2.8,0)(2.8,-1)

\pspolygon[style=fyp,fillcolor=light](6.8,-1)(6.8,0)(7.8,0)(7.8,-1)
\pspolygon[style=fyp,fillcolor=light](7.8,0)(7.8,1)(8.8,1)(8.8,0)
\pspolygon[style=fyp,fillcolor=light](7.8,-1)(7.8,0)(8.8,0)(8.8,-1)
\pspolygon[style=fyp,fillcolor=light](6.8,0)(6.8,1)(7.8,1)(7.8,0)
}
\end{pspicture}
\end{center}
\caption{Extremal convex stack polyominoes}\label{extremalstack}
\end{figure}
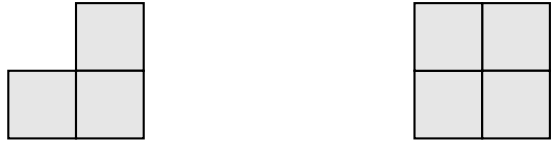
\end{Theorem}

\begin{proof}
It can be easily checked that $I_\MP$ is extremal Gorenstein, if $\MP$ is isomorphic to one of the two polyominoes shown in Figure~\ref{extremalstack}.

Conversely, assume that $I_\MP$ is extremal Gorenstein. Without loss of generality we  may assume that $[(1,1),(m,n)]$ is the smallest interval containing $V(\MP)$. Then Theorem~\ref{gorensteinstack} implies that $m=n$. Suppose first that $V(\MP)= [(1,1),(n,n)]$. Then, by  \cite[Theorem 4]{ERQ} of Ene, Rauf and Qureshi, it follows  that the regularity of $I_\MP$ is equal to $n$. Since $I_\MP$ is extremal Gorenstein, its regularity is equal to $3$. Thus $n=3$.

Next, assume that $V(\MP)$ is properly contained in $[(1,1),(n,n)]$. Since $I_\MP$ is linearly related, Corollary~\ref{corner} together with Theorem~\ref{gorensteinstack} imply that the top row of $\MP$ consists of only one cell and that $[(2,1),(n-1,n-1)]\subset V(\MP)$. Let $\MP'$ be  the polyomino induced by the rows $2,3,\ldots,n-1$ and the columns $1,2,\ldots, n-1$. Then $\MP'$ is the polyomino with  $V(\MP')=[(1,1),(n-2,n-1)]$. By applying again  \cite[Theorem~4]{ERQ} it follows that $\reg I_\MP'=n-2$. Corollary~\ref{inducedsubpolyomino} then implies that $\reg I_\MP\geq \reg I_{\MP'}= n-2$, and since $\reg I_\MP=3$ we deduce that $n\leq 5$. If $n=5$, then $I_\MP'$ is the ideal of $2$-minors of a $3\times 4$-matrix which has Betti numbers $\beta_{35}\neq 0$ and $\beta_{36}\neq 0$. Since $\MP'$ is an induced polyomino of $\MP$ and since $I_\MP$ is extremal Gorenstein, Corollary~\ref{inducedsubpolyomino} yields a contradiction.

Up  to isomorphism there exist for $n=4$ precisely the  Gorenstein polyominoes displayed  in Figure~\ref{width}. They are all not extremal Gorenstein as can be easily checked with CoCoA  or Singular. For $n=3$ any Gorenstein polyomino  is isomorphic to one of the two polyominoes shown in Figure~\ref{extremalstack}. This yields the desired conclusion.
\end{proof}

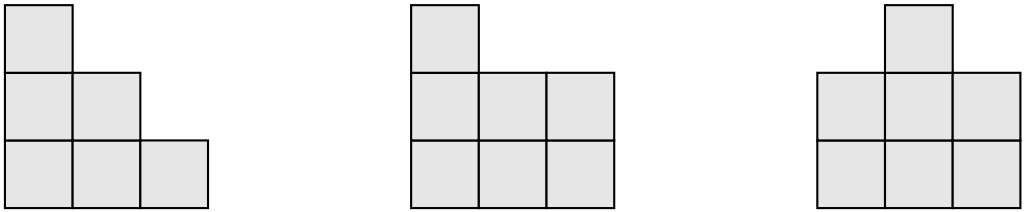
\begin{figure}[hbt]
\begin{center}
\psset{unit=0.9cm}
\begin{pspicture}(0.5,-0.5)(0.5,4)
\rput(-12,0){
%\pspolygon[style=fyp,fillcolor=light](4,0)(4,1)(5,1)(5,0)
\pspolygon[style=fyp,fillcolor=light](5,0)(5,1)(6,1)(6,0)
\pspolygon[style=fyp,fillcolor=light](6,0)(6,1)(7,1)(7,0)
\pspolygon[style=fyp,fillcolor=light](7,0)(7,1)(8,1)(8,0)
%\pspolygon[style=fyp,fillcolor=light](8,0)(8,1)(9,1)(9,0)
%\pspolygon[style=fyp,fillcolor=light](8,1)(8,2)(9,2)(9,1)
\pspolygon[style=fyp,fillcolor=light](5,1)(5,2)(6,2)(6,1)
\pspolygon[style=fyp,fillcolor=light](6,1)(6,2)(7,2)(7,1)
%\pspolygon[style=fyp,fillcolor=light](7,1)(7,2)(8,2)(8,1)
\pspolygon[style=fyp,fillcolor=light](5,2)(5,3)(6,3)(6,2)
%\pspolygon[style=fyp,fillcolor=light](6,2)(6,3)(7,3)(7,2)
%\pspolygon[style=fyp,fillcolor=light](7,2)(7,3)(8,3)(8,2)
%\pspolygon[style=fyp,fillcolor=light](6,3)(6,4)(7,4)(7,3)
}

\rput(-6,0){
%\pspolygon[style=fyp,fillcolor=light](4,0)(4,1)(5,1)(5,0)
\pspolygon[style=fyp,fillcolor=light](5,0)(5,1)(6,1)(6,0)
\pspolygon[style=fyp,fillcolor=light](6,0)(6,1)(7,1)(7,0)
\pspolygon[style=fyp,fillcolor=light](7,0)(7,1)(8,1)(8,0)
%\pspolygon[style=fyp,fillcolor=light](8,0)(8,1)(9,1)(9,0)
%\pspolygon[style=fyp,fillcolor=light](8,1)(8,2)(9,2)(9,1)
\pspolygon[style=fyp,fillcolor=light](5,1)(5,2)(6,2)(6,1)
\pspolygon[style=fyp,fillcolor=light](6,1)(6,2)(7,2)(7,1)
\pspolygon[style=fyp,fillcolor=light](7,1)(7,2)(8,2)(8,1)
\pspolygon[style=fyp,fillcolor=light](5,2)(5,3)(6,3)(6,2)
%\pspolygon[style=fyp,fillcolor=light](6,2)(6,3)(7,3)(7,2)
%\pspolygon[style=fyp,fillcolor=light](7,2)(7,3)(8,3)(8,2)
%\pspolygon[style=fyp,fillcolor=light](6,3)(6,4)(7,4)(7,3)
}

\rput(0,0){
%\pspolygon[style=fyp,fillcolor=light](4,0)(4,1)(5,1)(5,0)
\pspolygon[style=fyp,fillcolor=light](5,0)(5,1)(6,1)(6,0)
\pspolygon[style=fyp,fillcolor=light](6,0)(6,1)(7,1)(7,0)
\pspolygon[style=fyp,fillcolor=light](7,0)(7,1)(8,1)(8,0)
%\pspolygon[style=fyp,fillcolor=light](8,0)(8,1)(9,1)(9,0)
%\pspolygon[style=fyp,fillcolor=light](8,1)(8,2)(9,2)(9,1)
\pspolygon[style=fyp,fillcolor=light](5,1)(5,2)(6,2)(6,1)
\pspolygon[style=fyp,fillcolor=light](6,1)(6,2)(7,2)(7,1)
\pspolygon[style=fyp,fillcolor=light](7,1)(7,2)(8,2)(8,1)
\pspolygon[style=fyp,fillcolor=light](6,2)(6,3)(7,3)(7,2)
%\pspolygon[style=fyp,fillcolor=light](6,2)(6,3)(7,3)(7,2)
%\pspolygon[style=fyp,fillcolor=light](7,2)(7,3)(8,3)(8,2)
%\pspolygon[style=fyp,fillcolor=light](6,3)(6,4)(7,4)(7,3)
}

\end{pspicture}
\end{center}
\caption{Gorenstein  polyominoes of width 3}\label{width}
\end{figure}

The following theorem shows that besides of the two polyominoes listed in Theorem~\ref{stack}  whose polyomino ideal is extremal Gorenstein, there exist precisely two more
join-meet  ideals having this property.

\begin{Theorem}
\label{hibitwo}
Let $L = {\mathcal J}(P)$ be a simple finite distributive lattice.
Then the join-meet ideal $I_{L}$ is an extremal Gorenstein ideal
if and only if $L$ is one of the following displayed in Figure~\ref{Flattice}.

\begin{figure}[hbt]
\begin{center}
\psset{unit=0.3cm}
\begin{pspicture}(-25.3,-2.5)(4,5)
\rput(5,0){
\psline(-9,3)(-11,1)
%\psline(-9,3)(-9,1)
\psline(-9,3)(-7,1)
\psline(-11,1)(-9,-1)
%\psline(-9,1)(-9,-1)
\psline(-7,1)(-9,-1)
\psline(-9,3)(-11,5)
\psline(-13,3)(-11,5)
\psline(-13,3)(-11,1)
\psline(-11,1)(-13,-1)
\psline(-13,-1)(-11,-3)
\psline(-11,-3)(-9,-1)

\rput(-9,3){$\bullet$}
\rput(-11,1){$\bullet$}
%\rput(-9,1){$\bullet$}
\rput(-7,1){$\bullet$}
\rput(-9,-1){$\bullet$}
\rput(-11,-3){$\bullet$}
\rput(-13,-1){$\bullet$}
\rput(-13,3){$\bullet$}
\rput(-11,5){$\bullet$}
}

\rput(-5,0){
\psline(-9,3)(-11,1)
%\psline(-9,3)(-9,1)
%\psline(-9,3)(-7,1)
\psline(-11,1)(-9,-1)
%\psline(-9,1)(-9,-1)
%\psline(-7,1)(-9,-1)
\psline(-9,3)(-11,5)
\psline(-13,3)(-11,5)
\psline(-13,3)(-11,1)
\psline(-11,1)(-13,-1)
\psline(-13,-1)(-11,-3)
\psline(-11,-3)(-9,-1)

\rput(-9,3){$\bullet$}
\rput(-11,1){$\bullet$}
%\rput(-9,1){$\bullet$}
%\rput(-7,1){$\bullet$}
\rput(-9,-1){$\bullet$}
\rput(-11,-3){$\bullet$}
\rput(-13,-1){$\bullet$}
\rput(-13,3){$\bullet$}
\rput(-11,5){$\bullet$}
}

\rput(18,0){
\psline(-9,3)(-11,1)
%\psline(-9,3)(-9,1)
\psline(-9,3)(-7,1)
\psline(-11,1)(-9,-1)
%\psline(-9,1)(-9,-1)
\psline(-7,1)(-9,-1)
\psline(-9,3)(-11,5)
\psline(-13,3)(-11,5)
\psline(-13,3)(-11,1)
\psline(-11,1)(-13,-1)
\psline(-13,-1)(-11,-3)
\psline(-11,-3)(-9,-1)
\psline(-15,1)(-13,-1)
\psline(-15,1)(-13,3)

\rput(-9,3){$\bullet$}
\rput(-11,1){$\bullet$}
\rput(-15,1){$\bullet$}
\rput(-7,1){$\bullet$}
\rput(-9,-1){$\bullet$}
\rput(-11,-3){$\bullet$}
\rput(-13,-1){$\bullet$}
\rput(-13,3){$\bullet$}
\rput(-11,5){$\bullet$}
}

\rput(-16,0){
\rput(-11,-3){$\bullet$}
\rput(-13,-0.5){$\bullet$}
\rput(-9,-0.5){$\bullet$}
\rput(-11,-0.5){$\bullet$}
\rput(-13,2.5){$\bullet$}
\rput(-9,2.5){$\bullet$}
\rput(-11,2.5){$\bullet$}
\rput(-11,5){$\bullet$}
\psline(-11,-3)(-13,-0.5)
\psline(-11,-3)(-9,-0.5)
\psline(-11,-3)(-11,-0.5)
\psline(-11,5)(-11,2.5)
\psline(-9,-0.5)(-9,2.5)
\psline(-13,-0.5)(-13,2.5)
\psline(-13,-0.5)(-11,2.5)
\psline(-11,2.5)(-9,-0.5)
\psline(-11,-0.5)(-13,2.5)
\psline(-11,-0.5)(-9,2.5)
\psline(-11,5)(-13,2.5)
\psline(-11,5)(-9,2.5)

}
\end{pspicture}
\end{center}
\caption{}\label{Flattice}
\end{figure}
\end{Theorem}

\begin{proof}
Suppose that $L = {\mathcal J}(P)$ is simple and that $K[L]/I_{L}$ is Gorenstein.
it then follows that $P$ is pure and there is no element
$\xi \in P$ for which every $\mu \in P$ satisfies either $\mu \leq \xi$ or
$\mu \geq \xi$.  Since $h(L) = (1, q, 1, 0, \ldots, 0)$,
no $4$-element clutter is contained in $P$.

Suppose that a three-element clutter
$A$ is contained in $P$.
If none of the elements belonging to $A$ is a minimal element of $P$,
then, since $L = {\mathcal J}(P)$ is simple,
there exist at least two minimal elements.
Hence there exists a linear extension $\pi$ of $P$ with $|D(\pi)| \geq 3$,
a contradiction.  Thus at least one of the elements belonging to $A$
is a minimal element of $P$.  Similarly, at least one of the elements
belonging to $A$ is a maximal element.
Let an element $x \in A$ which is both minimal and maximal.  Then, since $P$ is pure,
one has $P = A$.  Let $A = \{ \xi_{1}, \xi_{2}, \xi_{3} \}$
with $A \neq P$, where $\xi_{1}$ is a minimal element
and $\xi_{2}$ is a maximal element.
Let $\mu_{1}$ be a maximal element with $\xi_{1} < \mu_{1}$ and
$\mu_{2}$ a minimal element with $\mu_{2} < \xi_{2}$.
Then neither $\mu_{1}$ nor $\mu_{2}$ belongs to $A$.
If $\xi_{3}$ is either minimal or maximal, then
there exists a linear extension $\pi$ of $P$ with $|D(\pi)| \geq 3$,
a contradiction.  Hence $\xi_{3}$ can be neither minimal nor maximal.
Then since $P$ is pure, there exist $\nu_{1}$ with
$\xi_{1} < \nu_{1} < \mu_{1}$ and $\nu_{2}$ with
$\mu_{2} < \nu_{2} < \xi_{2}$ such that $\{ \nu_{1}, \nu_{2}, \xi_{3}\}$
is a three-element clutter.
Hence there exists a linear extension $\pi$ of $P$ with $|D(\pi)| \geq 4$,
a contradiction.
Consequently, if $P$ contains a three-element clutter
$A$, then $P$ must coincide with $A$.
Moreover, if $P$ is a three-element clutter, then
$h(L) = (1,4,1)$ and $I_{L}$ is an extremal Gorenstein ideal.

Now, suppose that $P$ contains no clutter $A$ with $|A| \geq 3$.
Let a chain $C$ with $|C| \geq 3$ be contained in $P$.
Let $\xi, \, \xi'$ be the minimal elements of $P$ and
$\mu, \, \mu'$ the maximal elements of $P$
with $\xi < \mu$ and $\xi' < \mu'$.
Since $L = {\mathcal J}(P)$ is simple and since $P$ is pure,
it follows that there exist maximal chains
$\xi < \nu_{1} < \cdots < \nu_{r} < \mu$ and
$\xi' < \nu'_{1} < \cdots < \nu'_{r}  < \mu'$ such that
$\nu_{i} \neq \nu'_{i}$ for $1 \leq i \leq r$.
Then one has a linear extension $\pi$ of $P$ with $D(\pi) = 2 + r \geq 3$,
a contradiction.  Hence the cardinality of all maximal chains of $P$ is at most $2$.
However, if the cardinality of all maximal chains of $P$ is equal to $1$, then
$h(L) = (1,1)$.  Thus $I_{L}$ cannot be an extremal Gorenstein ideal.
If the cardinality of all maximal chains of $P$ is equal to $2$, then
$P$ is the  posets displayed in Figure~\ref{poset}. For each of them the join-meet ideal $I_{L}$ is an extremal Gorenstein ideal.

\begin{figure}[hbt]
\begin{center}
\psset{unit=0.3cm}
\begin{pspicture}(-25.3,-2.5)(4,5)
\rput(7,0){
\rput(-1,4){$\bullet$}
\rput(-4,4){$\bullet$}
\rput(-1,0){$\bullet$}
\rput(-4,0){$\bullet$}
\rput(-2.5,-2.5){$h(L) = (1,4,1)$}
\psline(-1,0)(-1,4)
\psline(-4,0)(-4,4)
}

\rput(-8,0){
\rput(-1,4){$\bullet$}
\rput(-4,4){$\bullet$}
\rput(-1,0){$\bullet$}
\rput(-4,0){$\bullet$}
\rput(-2.5,-2.5){$h(L) = (1,3,1)$}
\psline(-1,0)(-1,4)
\psline(-4,0)(-4,4)
\psline(-1,0)(-4,4)
}

\rput(-23,0){
\rput(-1,4){$\bullet$}
\rput(-4,4){$\bullet$}
\rput(-1,0){$\bullet$}
\rput(-4,0){$\bullet$}
\rput(-2.5,-2.5){$h(L) = (1,2,1)$}
\psline(-1,0)(-1,4)
\psline(-4,0)(-4,4)
\psline(-1,0)(-4,4)
\psline(-4,0)(-1,4)
}
\end{pspicture}
\end{center}
\caption{}\label{poset}
\end{figure}
\end{proof}

\end{document}